\documentclass[12pt,reqno]{amsart}
\usepackage[notref,notcite]{}
\usepackage{amssymb,verbatim}
\usepackage{amsmath}
\newcommand{\Rb}{\mathbb{R}}
\newcommand{\Beq}{\begin{equation}}
    \newcommand{\Eeq}{\end{equation}}

\newtheorem{theorem}{Theorem}[section]
\newtheorem{lemma}[theorem]{Lemma}
\newtheorem{proposition}[theorem]{Proposition}
\newtheorem{corollary}[theorem]{Corollary}
\newtheorem{definition}[theorem]{Definition}

\numberwithin{equation}{section}

\newcommand{\wh}{\widehat}

\newcommand{\rb}{\right)}
\newcommand{\lb}{\left(}

\usepackage[notref,notcite]{}
\usepackage{amssymb,verbatim}
\usepackage{amsmath}
\usepackage[usenames]{color}
\newcommand{\I}{\textsl{i}}
\newcommand{\md}{\mathrm{d}}

\newcommand{\vp}{\varphi}

\newcommand{\PD}{\partial}

\newcommand{\wt}{\widetilde}

\newcommand{\Jc}{\mathcal{J}}

\newcommand{\Sc}{\mathcal{S}}

\newcommand{\Sb}{\mathbb{S}}

\newcommand{\beq}{\begin{equation*}}
\newcommand{\eeq}{\end{equation*}}
\newcommand{\bal}{\begin{align}}
\newcommand{\eal}{\end{align}}

\renewcommand{\L}{\langle}
\newcommand{\R}{\rangle}

\newcommand{\bpr}{\begin{proof}}
    \newcommand{\epr}{\end{proof}}
\renewcommand{\o}{\omega}

\newcommand{\bel}[1]{\begin{equation}\label{#1}}
\newcommand{\ee}{\end{equation}}

\title[Range characterization]{Ray transform on Sobolev spaces of symmetric tensor fields, II: Range characterization }

\author[V. P. Krishnan]{Venkateswaran P. Krishnan}
\address{
Centre for Applicable Mathematics, Tata Institute of Fundamental Research, India }
\email{vkrishnan@tifrbng.res.in}

\author[V. A. Sharafutdinov]{Vladimir A. Sharafutdinov}
\address{Sobolev Institute of Mathematics, 4 Koptyug Av., 630090, Novosibirsk, Russia}
\email{sharaf@math.nsc.ru}

\begin{document}

\maketitle
\begin{abstract}
  The ray transform $I$ integrates symmetric $m$-tensor field in $\Rb^n$ over lines. This transform in Sobolev spaces was studied in our earlier work where higher order Reshetnyak formulas (isometry relations) were established. The main focus of the current work is the range characterization. In dimensions $n\geq 3$, the range characterization of the ray transform in Schwartz spaces is well-known; the main ingredient of the characterization is a system of linear differential equations of order $2(m+1)$ which are called John equations. Using the higher order Reshetnyak formulas, the range of the ray transform on Sobolev spaces is characterized in dimensions $n\geq 3$ in this paper.
\end{abstract}

\section{Introduction}

The ray transform integrates functions or more generally symmetric tensor fields over
lines  and the Radon transform integrates functions over hyperplanes in $\Rb^n$.   The ray
transform of functions is the main mathematical tool of computer tomography.
The ray transform of vector fields and that of symmetric second rank tensor fields arise in Doppler tomography
and travel time tomography, respectively. The Radon transform  originally arose while trying to decompose a solution of the wave equation in terms of plane waves. In 2-dimensions, the ray transform of functions coincides, up to notation, with the Radon transform.  The Radon transform is extensively studied in \cite{Helgason:Book,Natterer:Book} and a systematic treatment of the ray transform of symmetric tensor fields in Euclidean and Riemannian manifold settings is performed in \cite{Sharafutdinov:Book}. The ray and Radon transforms will be denoted by $I$ and $R$, respectively, in this paper.

Three questions naturally arise in the study of the Radon transform and ray transform of symmetric tensor fields
(or any other transform with tomographic applications) (i) inversion formulas (ii) stability
and (iii) range characterization. We do not discuss inversion formulas in this paper. Instead, we refer the interested reader to the aforementioned books \cite{Helgason:Book,Natterer:Book, Sharafutdinov:Book} where explicit inversion formulas are derived for such transforms. The second question is that of stability. As we will see, this is  related to the third question of range characterization; the main topic of this paper. For this reason, we begin with a discussion of stability for the Radon transform.
To make the discussion accessible to a broader audience, we have chosen to be informal in the Introduction. Our current work is a companion paper to our prior work on this topic \cite{HORF_Work}, and we largely follow the notation from that paper. To avoid repetition, we have chosen to give only a condensed version of the relevant material to make the discussion self-contained. We recommend the interested readers to read this in conjunction with our earlier work as and when required.

The Radon transform $R$ integrates a function along hyperplanes. The set of hyperplanes can be parameterized by points of $\Sb^{n-1}\times \Rb$. Then $R$ is defined by
    \[
    Rf(\xi,p) = \int\limits_{\L\xi,x\R =p} f(x) \, dx\quad\big((\xi,p)\in\Sb^{n-1}\times \Rb\big),
    \]
where $\langle \cdot\: ,\cdot\rangle$ is the standard dot-product in $\Rb^n$ and $dx$ is the $(n-1)$-dimensional Lebesgue measure on the hyperplane $\{ x\mid \L\xi,x\R =p\}$.  Some condition on $f$ should be imposed for the integral above to converge. The Reshetnyak formulas for the Radon transform is a family of isometry relations (the best stability estimates possible) involving certain weighted Sobolev spaces on $\Rb^n$ and on $\Sb^{n-1}\times \Rb$. For the precise definitions of the weighted Sobolev spaces appearing below, we refer the reader to \cite{Helgason:Book, Sh3, Sh5}. Section 2.1 in  \cite{HORF_Work} also provides a good summary of these works.  We have the following  \emph{family of higher order Reshetnyak formulas} for the Radon transform \cite{Sh5}:
\[
\lVert f\rVert_{H_{t}^{(r,s)}(\Rb^n)}= \lVert Rf\rVert_{H^{r, s+(n-1)/2}_{t+(n-1)/2}\lb \Sb^{n-1}\times \Rb\rb},
\]
for arbitrary reals $s,r$ and for $t>-n/2$. We discuss a few special cases of this formula. The case when $r=s=t=0$ is the original isometry relation between $f$ and its Radon transform, due to Reshetnyak in the 1960s.  The case when $r=0$, $s$ arbitrary and $t>-n/2$ was given by the second author in \cite{Sh3}.

The Reshetnyak formulas are closely related to the range characterization of the Radon transform. The classical Gel'fand-Helgason-Ludwig (GHL) range characterization theorem for the Radon transform on the Schwartz space is the following (below $\Sc(\Sb^{n-1}\times \Rb)$ denotes the Schwartz space of functions on $\Sb^{n-1}\times \Rb$):

A function $g\in \Sc(\Sb^{n-1}\times \Rb)$ is the Radon transform of a function $f\in \Sc(\Rb^n)$ if and only if the following two conditions are satisfied:

(1) $g(\o,p)=g(-\o,-p)$,

(2) for any integer $k\geq 0$, $\int_{\Rb} p^{k} g(\o,p)\,dp$ is a homogeneous polynomial of degree at most $k$ in $\o$.

A natural question is the range characterization of the Radon transform in weighted Sobolev spaces $H_{t}^{(r,s)}(\Rb^n)$. The following result is obtained in \cite{Sh5}: For arbitrary reals $r,s$ and $t>-n/2$,  the Radon transform
$R: \Sc(\Rb^n)\to \Sc(\Sb^{n-1}\times \Rb)$
uniquely extends to a bijective isometry of Hilbert spaces
\[
R:H^{(r,s)}_t(\Rb^n)\to H^{(r,s+(n-1)/2)}_{t+(n-1)/2,e}(\Sb^{n-1}\times \Rb).
\]
Hereafter, the index $e$ stands for ``even'', i.e., $H^{(r,s}_{t,e}(\Sb^{n-1}\times \Rb)$ is the subspace of $H^{(r,s}_t(\Sb^{n-1}\times \Rb)$ consisting of functions satisfying $g(\o,p)=g(-\o,-p)$.
We make one crucial remark here. The second part of the GHL condition disappears in the range characterization when one extends the Radon transform from the Schwartz space to the weighted Sobolev spaces.
The special case of this result, when $r=s=t=0$ can be found in \cite{Helgason:Book}, and for the case when $r=0$, $s$ arbitrary and $t>-n/2$ in \cite{Sh3}. To show the bijectivity of the Radon transform on these spaces, one proceeds as follows. The Reshetnyak formula shows that, being an isometry, the range of the operator $R|_{H^{(r,s)}_t(\Rb^n)}$ is a closed subspace of $H^{(r,s+(n-1)/2)}_{t+(n-1)/2,e}(\Sb^{n-1}\times \Rb)$, and the proof is completed by showing that the orthogonal complement of the latter subspace (with respect to the $H^{(r,s+(n-1)/2)}_{t+(n-1)/2}(\Sb^{n-1}\times \Rb)$ scalar product) is equal to zero.

We next discuss the existence of Reshetnyak formulas for the ray transform. Let $S^{m}\Rb^{n}$ be the  ${n+m-1}\choose m$-dimensional complex vector space of rank $m$ symmetric tensor fields on ${\Rb}^n$.
    The family of oriented straight lines in ${\mathbb R}^n$ is parameterized by points of the manifold
    \[
    T{\mathbb S}^{n-1}=\{(x,\xi)\in{\mathbb R}^n\times{\mathbb R}^n\mid |\xi|=1,\langle x,\xi\rangle=0\}\subset{\mathbb R}^n\times{\mathbb R}^n,
    \]
    that is, the tangent bundle of the unit sphere ${\mathbb S}^{n-1}$. A point $(x,\xi)\in T{\mathbb S}^{n-1}$ determines the line $\{x+t\xi\mid t\in{\mathbb R}\}$. Let $\Sc(\Rb^{n}; S^{m}\Rb^{n})$ denote the Schwartz space of $S^{m}\Rb^{n}$-valued functions on $\Rb^n$ and $\Sc (T\Sb^{n-1})$ denote the Schwartz space of functions on $T\Sb^{n-1}$.
    The {\it ray transform} $I$ is the linear continuous operator
\begin{equation}
        I: \Sc(\Rb^{n}; S^{m}\Rb^{n})\to \Sc (T\Sb^{n-1})
                                         \label{1.1}
\end{equation}
    that is defined, for $f=(f_{i_1\dots i_m})\in\Sc(\Rb^{n}; S^{m}\Rb^{n})$, by
\begin{equation}
If (x,\xi)=\int\limits_{-\infty}^\infty f_{i_1\dots i_m}(x+t\xi)\,\xi^{i_1}\dots\xi^{i_m}\,\md t=\int\limits_{-\infty}^\infty \L f(x+t\xi),\xi^m\R\,d t\quad\big((x,\xi)\in T{\Sb}^{n-1}\big).
                                         \label{1.2}
\end{equation}

 Here and henceforth, we use the Einstein summation rule: the summation from 1 to $n$ is assumed over every index repeated in lower and upper positions in a monomial. We use either lower or upper indices for denoting coordinates of vectors and tensors.  Since we work in Cartesian coordinates only, there is no difference between covariant and contravariant tensors. For $m>0$, the ray transform $I$ has an infinite dimensional kernel consisting of so-called potential tensor fields. Therefore it is natural to restrict the ray transform to the orthogonal complement of potential tensor fields, called solenoidal tensor fields. Analogous to  the higher order Reshetnyak formulas for the Radon transform, the following family of higher order Reshetnyak formulas for the ray transform of symmetric tensor fields was derived recently in our work \cite{HORF_Work}: For an integer $r\geq 0$, real $s$ and $t>-n/2$, the $r^{\mathrm{th}}$ order Reshetnyak formula for the ray transform restricted to  solenoidal tensor fields is
 \Beq
 \lVert f\rVert_{H^{(r,s)}_t(\Rb^n; S^{m}\Rb^n)}=\lVert  If\rVert_{H^{\lb r, s+1/2\rb}_{t+1/2}(T\Sb^{n-1})}.
                                    \label{1.3}
 \Eeq
Definitions of the spaces in the formula are given in the next section.
 Two special cases of this Reshetnyak formula were derived by the second author. The first was for the case $r=0$ in \cite{Sh3}, and a first order Reshetnyak formula corresponding to $r=1$ (with a slightly different definition for the weighted Sobolev spaces) for the case of the ray transform of functions was given in \cite{Sh5}.

 Now we come to the question of range characterization for the ray transform.  John \cite{J} studied the range characterization of the ray transform of Schwartz class functions in $\Rb^3$ and showed that a function is in the range of this transform if and only if it solves a certain second order linear differential equation, the so-called {\it John equation}. Helgason \cite{Helgason_Acta_Paper} generalized this result to  an arbitrary $n\geq 3$ dimension. A system of linear second order differential equations appears in the case of $n\ge4$, they are still called John equations. The second author proved a range characterization result for the ray transform of symmetric tensor fields in Schwartz space in dimensions $n\geq 3$ \cite[Theorem 2.10.1]{Sharafutdinov:Book}.

For the $m$-tensor case, a big system of linear differential equations of order $2(m+1)$ appears, see equations \eqref{2.6} below. Observe that the left-hand side of every equation \eqref{2.6} is actually a product of $m$ John operators defined by \eqref{2.7}. Equations \eqref{2.6} are not independent, there are many linear relations between them since John operators commute with each other. One would like to distinguish a minimal subsystem of \eqref{2.6} which is still equivalent to the whole system. We do not try to minimize the system \eqref{2.6} in the present paper. Nevertheless, let us mention that, in the 3-dimensional case, the system \eqref{2.6} is equivalent to one linear differential equation of order $2(m+1)$ if an appropriate parametrization of the family of lines in ${\mathbb R}^3$ has been chosen \cite{NSV}. In an arbitrary dimension, minimization of the system \eqref{2.6} is discussed in \cite{De}.

Integral GHL conditions for the Radon transform are very different of John differential equations for the ray transform. Nevertheless, they are related as is shown by Denisiuk \cite{De} by restricting the ray transform to 2D planes.

Unlike GHL conditions, the John equations persist in the passage from the Schwartz space to weighted Sobolev spaces. In the scalar case, this fact is proved in \cite{Sh4}. The main goal of the present article is the proof of the latter fact for the ray transform of symmetric tensor fields. Namely, being understood in the distribution sense, the John equations characterize the range of the ray transform on the spaces $H^{(r,s)}_t(\Rb^n; S^{m}\Rb^n)$. Unlike the case of the Schwartz space, Reshetnyak formulas play an important role in the characterization.

In our opinion, the difference between GHL conditions  for the Radon transform and John equations for the ray transform is caused by the following. The ray transform  is over-determined in dimensions $n\geq 3$: we try to recover a tensor field depending on $n$ variables from a $(2n-2)$-dimensional set of information. On the other hand, in the case of the Radon transform, the problem of recovering $f$ from $Rf$ is formally determined.

Finally, we shortly discuss the ray transform in the 2-dimensional case. As we have mentioned, the ray transform of functions coincides with the Radon transform in 2-dimensions. For symmetric $m$-tensor fields, the range characterization of the 2-dimensional ray transform in the Scwartz space was obtained by Pantjukhina \cite{P}; some version of integral GHL conditions participates here. The range characterization of the 2D ray transform in weighted Sobolev spaces of symmetric $m$-tensor fields is not studied so far for $m>0$. The following interesting question remains open: What kind of conditions, either John's differential equations or GHL integral conditions, should be used in the latter problem?

\section{Preliminaries and statement of the main result}

In this section, we state all preliminary definitions and results required for a precise statement of the range characterization.

\subsection{Ray transform of symmetric tensor fields}

Recall from the previous section that $\Sc(\Rb^n; S^{m}\Rb^n)$ is the Schwartz space of symmetric $m$-tensor valued functions on $\Rb^n$ (each component of the tensor field is a Schwartz class function) and we  also introduced the Schwartz space ${\mathcal S}(T{\mathbb S}^{n-1})$
that is defined as follows. Given a function $\varphi\in C^\infty(T{\mathbb S}^{n-1})$, we extend it to some neighborhood of $T{\mathbb S}^{n-1}$ in
${\mathbb R}^n\times{\mathbb R}^n$ so that (the extension is again denoted by $\varphi$)
$$
\varphi(x,r\xi)=\varphi(x,\xi)\ (r>0),\quad \varphi(x+r\xi,\xi)=\varphi(x,\xi)\ (r\in{\mathbb R}).
$$
We say that a function $\varphi\in C^\infty(T{\mathbb S}^{n-1})$ belongs to ${\mathcal S}(T{\mathbb S}^{n-1})$ if the seminorm
$$
\|\varphi\|_{k,\alpha,\beta}=\sup\limits_{(x,\xi)\in T{\mathbb S}^{n-1}}
\left|(1+|x|)^k\partial^\alpha_x\partial^\beta_\xi\varphi(x,\xi)\right|
$$
is finite for every $k\in{\mathbb N}$ and for all multi-indices $\alpha$ and $\beta$. The family of these seminorms defines the topology on
${\mathcal S}(T{\mathbb S}^{n-1})$.

Recall that the ray transform \eqref{1.1} is defined by \eqref{1.2}.
In the case of an even $m$, the ray transform is the linear continuous operator
$$
I:{\mathcal S}({\Rb}^n;S^m{\Rb}^n)\rightarrow {\mathcal S}_e(TS^{n-1}),
$$
and in the case of an odd $m$,
$$
I:{\mathcal S}({\Rb}^n;S^m{\Rb}^n)\rightarrow {\mathcal S}_o(TS^{n-1}),
$$
where ${\mathcal S}_e(TS^{n-1})\ ({\mathcal S}_o(TS^{n-1}))$ is the subspace of ${\mathcal S}_e(TS^{n-1})$ consisting of functions satisfying $\varphi(x,-\xi)=\varphi(x,\xi)$ (satisfying $\varphi(x,-\xi)=-\varphi(x,\xi)$).
To unite these formulas, let us introduce the {\it parity} of $m$
$$
\pi(m)=\left\{\begin{array}{l}e\ \mbox{if}\ m\ \mbox{is even},\\ o\ \mbox{if}\ m\ \mbox{is odd}.\end{array}\right.
$$
Then the ray transform can be initially considered as a linear continuous operator
\begin{equation}
I:{\mathcal S}({\Rb}^n;S^m{\Rb}^n)\rightarrow {\mathcal S}_{\pi(m)}(TS^{n-1}).
                          \label{2.2}
\end{equation}

\subsection{Fourier slice theorem}
The Fourier transform of symmetric tensor fields
$$
F:\Sc(\Rb^n; S^m\Rb^n)\to\Sc(\Rb^n; S^m\Rb^n),\quad f\mapsto \wh f
$$
is defined component wise (hereafter $\I$ is the imaginary unit):
\[
\wh{f}_{i_1 \cdots i_{m}}(y)=\frac{1}{(2\pi)^{n/2}}\int \limits_{\Rb^n} e^{-\I \langle y,x\rangle} f_{i_1 \cdots i_m}(x) \, d x.
\]
The Fourier transform $F:\Sc(T\Sb^{n-1})\to\Sc(T\Sb^{n-1}),\ \vp\mapsto\wh\vp$ is defined as the $(n-1)$-dimensional Fourier transform over the subspace $\xi^{\perp}$:
    \[
   \wh{\vp}(y,\xi)=\frac{1}{(2\pi)^{(n-1)/2}}\int\limits_{\xi^{\perp}} e^{-\I \L y,x\R} \vp(x,\xi) \, d x\quad\big((y,\xi)\in T{\Sb}^{n-1}\big).
    \]
The Fourier slice theorem \cite[formula (2.1.5)]{Sharafutdinov:Book} states:
 \Beq
 \wh{If}(y,\xi)=\sqrt{2\pi}\L \wh{f}(y), \xi^{m}\R \mbox{ for } (y,\xi)\in T\Sb^{n-1}.
                                       \label{2.3}
 \Eeq

\subsection{Range characterization of the ray transform on the Schwartz space}

We first observe that the right-hand side of \eqref{1.2} makes sense for all $(x,\xi)\in \Rb^n\times \Rb^n\setminus \{0\}$, and we use notation $J$ to denote the extended operator. Strictly speaking, we define the continuous linear operator
$$
J: \Sc(\Rb^{n}; S^{m}\Rb^{n})\to C^\infty(\Rb^n\times \Rb^n\setminus \{0\})
$$
by
\begin{equation}
J f(x,\xi)= \int\limits_{-\infty}^\infty f_{i_1\dots i_m}(x+t\xi)\,\xi^{i_1}\dots\xi^{i_m}\, d t\quad
(x\in{\mathbb R}^n,0\neq\xi\in{\mathbb R}^n).
                                \label{2.4}
\end{equation}
This extension allows us to take partial derivatives of $Jf$ with respect to variables $x$ and $\xi$ freely. One can go back and forth between the original definition and the extended one by
\[
Jf|_{T\Sb^{n-1}}=If \quad \mbox{and} \quad Jf(x,\xi)=|\xi|^{m-1}If\lb x-\frac{\L x,\xi\R\xi}{|\xi|^2}, \frac{\xi}{|\xi|}\rb.
\]
Given a function $\vp\in T\Sb^{n-1}$, let us define an extension $\psi$ of this function to $\Rb^n\times \Rb^n\setminus \{0\}$ as follows:
\Beq
\psi(x,\xi)=|\xi|^{m-1}\vp\lb x-\frac{\L x,\xi\R\xi}{|\xi|^2}, \frac{\xi}{|\xi|}\rb.
                                   \label{2.5}
\Eeq

The range characterization theorem for the ray transform of symmetric tensor fields in dimension $n\geq 3$ is as follows:

\begin{theorem}  \label{Th2.1}
A function $\varphi\in{\mathcal S}(T{\mathbb S}^{n-1})\ (n\geq3)$ belongs to the range of the operator \eqref{1.1} if and only if the following two conditions hold:
\begin{enumerate}
\item[(1)] $\varphi(x,-\xi)=(-1)^m\varphi(x,\xi)$;

\item[(2)] the function $\psi\in C^\infty\big({\mathbb R}^n\times({\mathbb R}^n\setminus\{0\})\big)$, defined by \eqref{2.5}, satisfies the
equations
\begin{equation}
\Big(\frac{\partial^2}{\partial x^{i_1}\partial\xi^{j_1}}-\frac{\partial^2}{\partial x^{j_1}\partial\xi^{i_1}}\Big)\dots
\Big(\frac{\partial^2}{\partial x^{i_{m+1}}\partial\xi^{j_{m+1}}}-\frac{\partial^2}{\partial x^{j_{m+1}}\partial\xi^{i_{m+1}}}\Big)\psi=0
                                                   \label{2.6}
\end{equation}
for all indices $1\leq i_1,j_1,\dots,i_{m+1},j_{m+1}\leq n$.
\end{enumerate}
\end{theorem}

We call \eqref{2.6} the {\it John equations} and the differential operators
\begin{equation}
J_{ij}=\frac{\partial^2}{\partial x^i\partial\xi^j}-\frac{\partial^2}{\partial x^j\partial\xi^i}:
C^\infty({\Rb}^n\times{\Rb}^n)\rightarrow C^\infty({\Rb}^n\times{\Rb}^n),
                                                   \label{2.7}
\end{equation}
the {\it John operators}. (We hope the reader is not confused by similarity between the notation $J$ for the operator \eqref{2.4} and $J_{ij}$ for John operators; both the notations are standard ones.)
In the case of $(m,n)=(0,3)$, Theorem \ref{Th2.1} is equivalent to John's result \cite{J}. In the case of $m=0$, Theorem \ref{Th2.1} was proved by Helgason \cite{Helgason_Acta_Paper}. For the proof in the general case see \cite[Theorem 2.10.1]{Sharafutdinov:Book}.

Our goal, as already mentioned, is to generalize the above result to weighted Sobolev spaces. The spaces are discussed in next few paragraphs.

\subsection{The vector fields $X_i$ and $\Xi_i$ and intrinsic John operators}
The following first order differential operators on $T\Sb^{n-1}$ were introduced in
\cite{{Krishnan:Manna:Sahoo:Sharafutdinov:2019},{Krishnan:Manna:Sahoo:Sharafutdinov:2020}}. Consider $\Rb^n\times \Rb^n$ with variables $(x,\xi)$  and introduce the following vector fields:
$$
\begin{aligned}
\wt{X}_i&= \frac{\PD}{\PD x_i}-\xi_i \xi^p \frac{\PD}{\PD x^p},\\
\wt{\Xi}_i& = \frac{\PD}{\PD \xi_i} -x_i \xi^p \frac{\PD}{\PD x^p} - \xi_i\xi^p \frac{\PD}{\PD \xi^p}.
\end{aligned}
$$
These vector fields are tangent to $T\Sb^{n-1}$  at every point  $(x,\xi)\in T\Sb^{n-1}$, as was shown in
\cite{Krishnan:Manna:Sahoo:Sharafutdinov:2019},  and therefore can be viewed as vector fields on $T\Sb^{n-1}$. Let $X_i$ and $\Xi_i$ be the restrictions of these vector fields to $T\Sb^{n-1}$.

The operators $X_i$ and $\Xi_i$ are related to the Fourier transform by the equalities \cite{Krishnan:Manna:Sahoo:Sharafutdinov:2019}:
$$
        \widehat{X_i\varphi}=\textsl{i}\,y_i\,\wh{\vp},\quad
        \widehat{\Xi_i\varphi}=\Xi_i\widehat\varphi,\quad
$$
which hold for every function $\varphi\in{\mathcal S}(T{\Sb}^{n-1})$ and for every $1\leq i\leq n$.

We use the vector fields $X_i$ and $\Xi_i$ to introduce the second order differential operators
$$
{\mathcal J}_{ij}:C^\infty(T{\Sb}^{n-1})\rightarrow C^\infty(T{\Sb}^{n-1})\quad(1\le i,j\le n)
$$
by
${\mathcal J}_{ij}=X_i\Xi_j-X_j\Xi_i$.
The operator $\Jc_{ij}$ are called {\it intrinsic John operators}.

\subsection{The spaces $H^{(r,s)}_t(T\Sb^{n-1})$}   \label{S:Xi}

We define the second order differential operator $\Delta_\xi$ on $T\Sb^{n-1}$ by
$\Delta_\xi=-\sum\limits_{i=1}^{n} \Xi_i^2$.
This operator is used for defining the weighted Sobolev spaces on $T\Sb^{n-1}$.
The following scalar product was introduced in \cite{Sh3}. For $\varphi_j \in{\mathcal S}(T{\Sb}^{n-1})\ (j=1,2)$,
\begin{equation}
    (\varphi_1,\varphi_2)_{H^s_t(T{\Sb}^{n-1})}=\frac{\Gamma\big(\frac{n-1}{2}\big)}{4\pi^{(n+1)/2}}
    \int\limits_{{\Sb}^{n-1}}\int\limits_{\xi^\bot}|y|^{2t}(1+|y|^2)^{s-t}\wh{\varphi_1}(y,\xi)\,\overline{\wh{\varphi_2}(y,\xi)}\, d y d \xi.
                                        \label{2.8}
\end{equation}
We showed in \cite{HORF_Work} that \eqref{1.1} is a positive semi-definite operator with respect to the scalar product \eqref{2.8} for any real $s$ and $t > -(n - 1)/2$.

\begin{definition} \label{D2.1}
\cite[Def 3.4]{HORF_Work}
For an integer $r\geq 0$, real $s$ and $t>-(n-1)/2$, introduce the norm on $\Sc(T\Sb^{n-1})$
 \begin{equation}
         \|\varphi\|^2_{H^{(r,s)}_{t}(T\Sb^{n-1})}=\big((\mathrm{Id}+\Delta_\xi)^r\varphi,\varphi\big)_{H^s_t(T{\Sb}^{n-1})}
 =\sum\limits_{l=0}^r{r\choose l}\big(\Delta_\xi^l\varphi,\varphi\big)_{H^s_t(T{\Sb}^{n-1})}
                                                 \label{2.9}
 \end{equation}
and define the Hilbert space $H^{(r,s)}_t(T{\Sb}^{n-1})$ as the completion of ${\mathcal S}(T{\Sb}^{n-1})$ with respect to the norm \eqref{2.9}.
\end{definition}

The Hilbert spaces $H^{(r,s)}_{t}(\Rb^n;S^{m}\Rb^n)$ are defined in such a way that there is the equality in \eqref{1.3}; see \cite[Def 5.2]{HORF_Work}. We do not oresent the precise definition of $H^{(r,s)}_{t}(\Rb^n;S^{m}\Rb^n)$ here since it does not participate in calculations related to the range characterization result below.

Recall that ${\mathcal S}_{\mathrm sol}({\Rb}^n;S^m{\Rb}^n)$ is the subspace of ${\mathcal S}({\Rb}^n;S^m{\Rb}^n)$ consisting of {\it solenoidal} tensor fields satisfying
\begin{equation}
\sum\limits_{p=1}^n\frac{\partial f_{pi_2\dots i_m}}{\partial x^p}=0.
                                        \label{2.10}
\end{equation}
By \cite[Theorem 1.1]{HORF_Work}, for all $n\ge2$ and $m\ge0$, the ray transform
$$
I:{\mathcal S}({\Rb}^n;S^m{\Rb}^n)\rightarrow {\mathcal S}_{\pi(m)}(T{\Sb}^{n-1})
$$
 extends to the isometric embedding of Hilbert spaces
\begin{equation}
I:H^{(r,s)}_{t,\mathrm{sol}}({\Rb}^n;S^m{\Rb}^n)\rightarrow H^{(r,s+1/2)}_{t+1/2,\pi(m)}(T{\Sb}^{n-1})
                                        \label{2.11}
\end{equation}
for every integer $r\ge0$, every real $s$ and every $t>-n/2$. See \eqref{2.2} for the additional index $\pi(m)$ on the right-hand side of \eqref{2.11}. Therefore the range of the operator \eqref{2.11} is a closed subspace of $H^{(r,s+1/2)}_{t+1/2,\pi(m)}(T{\Sb}^{n-1})$.

Let ${\mathcal S}'(T{\Sb}^{n-1})$ be the space of tempered distributions on $T{\Sb}^{n-1}$. By $\L\varphi\vert\psi\R$ we denote the value of a distribution $\varphi\in{\mathcal S}'(T{\Sb}^{n-1})$ on a test function $\psi\in{\mathcal S}(T{\Sb}^{n-1})$. The embedding
${\mathcal S}(T{\Sb}^{n-1})\subset{\mathcal S}'(T{\Sb}^{n-1})$ is defined by
\begin{equation}
\L\varphi\vert\psi\R=\int\limits_{{\Sb}^{n-1}}\int\limits_{\xi^\bot}\varphi(x,\xi)\psi(x,\xi)\,dxd\xi\quad\mbox{for}\quad
\varphi,\psi\in{\mathcal S}(T{\Sb}^{n-1}).
                         \label{2.12}
\end{equation}

\begin{proposition} \label{P2.1}
For an integer $r\ge0$, real $s$ and $t\in\big(-(n-1)/2,(n-1)/2\big)$, the space $H^{(r,s)}_t(T{\Sb}^{n-1})$ consists of tempered distributions. More precisely, the identity map of
${\mathcal S}(T{\Sb}^{n-1})$ extends to a continuous embedding
$H^{(r,s)}_t(T{\Sb}^{n-1})\subset{\mathcal S}'(T{\Sb}^{n-1})$.
\end{proposition}

\begin{proof}
There exists the continuous embedding
$H^{(r,s)}_t(T{\Sb}^{n-1})\subset H^s_t(T{\Sb}^{n-1})$. Therefore it suffices to prove Proposition \ref{P2.1} in the case of $r=0$, i.e., to prove that the identical map of ${\mathcal S}(T{\Sb}^{n-1})$ extends to a continuous embedding
$H^s_t(T{\Sb}^{n-1})\subset{\mathcal S}'(T{\Sb}^{n-1})$. To this end we will demonstrate the validity of the estimate
$$
|\L\varphi\vert\psi\R|\le C_{s,t}\|\varphi\|_{H^s_t(T{\Sb}^{n-1})}\|\psi\|_{H^{-s}_{-t}(T{\Sb}^{n-1})}\quad\mbox{for}\quad
\varphi,\psi\in{\mathcal S}(T{\Sb}^{n-1}).
$$

With the help of the Plancherel formula, \eqref{2.12} can be written as
$$
\L\varphi\vert\psi\R=\int\limits_{{\Sb}^{n-1}}\int\limits_{\xi^\bot}
\widehat\varphi(y,\xi)\widehat\psi(y,\xi)\,dyd\xi.
$$
Writing this in the form
$$
\L\varphi\vert\psi\R=\int\limits_{{\Sb}^{n-1}}\int\limits_{\xi^\bot}
\Big(|y|^t(1+|y|^2)^{(s-t)/2}\widehat\varphi(y,\xi)\Big)\Big(|y|^{-t}(1+|y|^2)^{-(s-t)/2}\widehat\psi(y,\xi)\Big)\,dyd\xi
$$
and applying the Cauchy -- Bunyakovski inequality, we obtain
$$
\begin{aligned}
|\L\varphi\vert\psi\R|&\le\Big(\int\limits_{{\Sb}^{n-1}}\int\limits_{\xi^\bot}
|y|^{2t}(1+|y|^2)^{s-t}|\widehat\varphi(y,\xi)|^2\,dyd\xi\Big)^{1/2}\times\\
&\times\Big(\int\limits_{{\Sb}^{n-1}}\int\limits_{\xi^\bot}|y|^{-2t}(1+|y|^2)^{t-s}|\widehat\psi(y,\xi)|^2\,dyd\xi\Big)^{1/2}
=C_{s,t}\|\varphi\|_{H^s_t(T{\Sb}^{n-1})}\|\psi\|_{H^{-s}_{-t}(T{\Sb}^{n-1})}.
\end{aligned}
$$
\end{proof}

By Proposition \ref{P2.1}, 
we have the linear continuous operator
\begin{equation}
\begin{aligned}
&{\mathcal J}_{i_1j_1}\Big({\mathcal J}_{i_2j_2}-(\xi_{i_2}X_{j_2}-\xi_{j_2}X_{i_2})\Big)
\Big({\mathcal J}_{i_3j_3}-2(\xi_{i_3}X_{j_3}-\xi_{j_3}X_{i_3})\Big)\dots\\
&\dots\Big({\mathcal J}_{i_{m+1}j_{m+1}}-m(\xi_{i_{m+1}}X_{j_{m+1}}-\xi_{j_{m+1}}X_{i_{m+1}})\Big)
:H^{(r,s+1/2)}_{t+1/2,\pi(m)}(T{\Sb}^{n-1})\rightarrow {\mathcal S}'(T{\Sb}^{n-1})
\end{aligned}
                          \label{2.13}
\end{equation}
under the only restriction $t\in\big(-n/2,(n-2)/2\big)$.

We are now ready to state the main result of this paper.

\begin{theorem}[{\rm Main theorem}] \label{Th2.2}
Let $n\ge3$ and $t\in\big(-(n-1)/2,(n-2)/2\big)$. For every integers $m\ge0,r\ge0$ and every real $s$, the range of the operator \eqref{2.11} coincides with the intersection of kernels of differential operators \eqref{2.13} for all $1\le i_1,j_1,\dots,i_{m+1},j_{m+1}\le n$. In other words, a ``function'' $\varphi\in H^{(r,s+1/2)}_{t+1/2,\pi(m)}(T{\Sb}^{n-1})$ belongs to the range of the operator \eqref{2.11} if and only if it satisfies the equations
\begin{equation}
\begin{aligned}
&{\mathcal J}_{i_1j_1}\Big({\mathcal J}_{i_2j_2}-(\xi_{i_2}X_{j_2}-\xi_{j_2}X_{i_2})\Big)
\Big({\mathcal J}_{i_3j_3}-2(\xi_{i_3}X_{j_3}-\xi_{j_3}X_{i_3})\Big)\dots\\
&\dots\Big({\mathcal J}_{i_{m+1}j_{m+1}}-m(\xi_{i_{m+1}}X_{j_{m+1}}-\xi_{j_{m+1}}X_{i_{m+1}})\Big)\varphi=0
\end{aligned}
                          \label{2.14}
\end{equation}
for all $1\le i_1,j_1,\dots,i_{m+1},j_{m+1}\le n$. Equations \eqref{2.14} are understood in the distribution sense.
\end{theorem}

\section{Range characterization for ray transform of Schwartz class tensor fields using intrinsic John operators}
We start with the following lemma.

\begin{lemma} \label{L3.1}
Assume a function $\psi\in C^\infty\big({\Rb}^n\times({\Rb}^n\setminus\{0\})\big)$ to be positively homogeneous of degree $\lambda$ in $\xi$ and to satisfy
\begin{equation}
\psi(x+t\xi,\xi)=\psi(x,\xi)\quad(t\in{\Rb}).
                          \label{3.1}
\end{equation}
Let $\varphi\in C^\infty(T{\Sb}^{n-1})$ be the restriction of $\psi$ to $T{\Sb}^{n-1}$. Then the equality
\begin{equation}
(J_{ij}\psi)|_{T{\Sb}^{n-1}}=\big({\mathcal J}_{ij}-(\lambda+1)(\xi_iX_j-\xi_jX_i)\big)\varphi
                          \label{3.2}
\end{equation}
holds for any $1\le i,j\le n$.
\end{lemma}

\begin{proof}
(Compare with \cite[Lemma 2.4]{Sh4})
By the very definition  of the operators
$X_i$
and
$\Xi_i$,
the equality
\begin{equation}
{\mathcal J}_{ij}\varphi=(X_i\Xi_j-X_j\Xi_i)\varphi=({\tilde X}_i{\tilde\Xi}_j-{\tilde X}_j{\tilde\Xi}_i)\psi
                                         \label{3.3}
\end{equation}
holds on
$T{\Sb}^{n-1}$.

By the definition of the operators
${\tilde X}_i$
and
${\tilde\Xi}_i$,
$$
\begin{aligned}
{\tilde X}_i{\tilde\Xi}_j-{\tilde X}_j{\tilde\Xi}_i&=
\Big(\frac{\partial}{\partial x^i}-\xi_i\xi^p\frac{\partial}{\partial x^p}\Big)
\Big(\frac{\partial}{\partial\xi^j}-x_j\xi^q\frac{\partial}{\partial x^q}-\xi_j\xi^q\frac{\partial}{\partial \xi^q}\Big)\\
&-\Big(\frac{\partial}{\partial x^j}-\xi_j\xi^p\frac{\partial}{\partial x^p}\Big)
\Big(\frac{\partial}{\partial\xi^i}-x_i\xi^q\frac{\partial}{\partial x^q}-\xi_i\xi^p\frac{\partial}{\partial \xi^q}\Big).
\end{aligned}
$$
After opening parentheses, this becomes
\begin{equation}
\begin{aligned}
{\tilde X}_i{\tilde\Xi}_j-{\tilde X}_j{\tilde\Xi}_i&=
\Big(\frac{\partial^2}{\partial x^i\partial\xi^j}-\frac{\partial^2}{\partial x^j\partial\xi^i}\Big)
+\Big(x_i\xi^p\frac{\partial^2}{\partial x^j\partial x^p}-x_j\xi^p\frac{\partial^2}{\partial x^i\partial x^p}\Big)\\
&+\Big(\xi_i\xi^p\frac{\partial^2}{\partial x^j\partial\xi^p}-\xi_j\xi^p\frac{\partial^2}{\partial x^i\partial\xi^p}\Big)
-\Big(\xi_i\xi^p\frac{\partial^2}{\partial x^p\partial\xi^j}-\xi_j\xi^p\frac{\partial^2}{\partial x^p\partial\xi^i}\Big)\\
&-(x_i\xi_j-x_j\xi_i)\xi^p\xi^q\frac{\partial^2}{\partial x^p\partial x^q}.
\end{aligned}
                                         \label{3.4}
\end{equation}
We transform the fourth term on the right-hand side as follows:
$$
\begin{aligned}
\xi_i\xi^p\frac{\partial^2}{\partial x^p\partial\xi^j}-\xi_j\xi^p\frac{\partial^2}{\partial x^p\partial\xi^i}
&=\xi_i\frac{\partial}{\partial\xi^j}\Big(\xi^p\frac{\partial}{\partial x^p}\Big)-\xi_i\frac{\partial}{\partial x^j}
-\xi_j\frac{\partial}{\partial\xi^i}\Big(\xi^p\frac{\partial}{\partial x^p}\Big)+\xi_j\frac{\partial}{\partial x^i}\\
&=\Big(\xi_i\frac{\partial}{\partial\xi^j}-\xi_j\frac{\partial}{\partial\xi^i}\Big)\xi^p\frac{\partial}{\partial x^p}
-\Big(\xi_i\frac{\partial}{\partial x^j}-\xi_j\frac{\partial}{\partial x^i}\Big).
\end{aligned}
$$
Substitute this expression into \eqref{3.4}
$$
\begin{aligned}
{\tilde X}_i{\tilde\Xi}_j-{\tilde X}_j{\tilde\Xi}_i&=
\Big(\frac{\partial^2}{\partial x^i\partial\xi^j}-\frac{\partial^2}{\partial x^j\partial\xi^i}\Big)
+\Big(x_i\xi^p\frac{\partial^2}{\partial x^j\partial x^p}-x_j\xi^p\frac{\partial^2}{\partial x^i\partial x^p}\Big)\\
&+\Big(\xi_i\xi^p\frac{\partial^2}{\partial x^j\partial\xi^p}-\xi_j\xi^p\frac{\partial^2}{\partial x^i\partial\xi^p}\Big)\\
&-\Big(\xi_i\frac{\partial}{\partial\xi^j}-\xi_j\frac{\partial}{\partial\xi^i}\Big)\xi^p\frac{\partial}{\partial x^p}
+\Big(\xi_i\frac{\partial}{\partial x^j}-\xi_j\frac{\partial}{\partial x^i}\Big)\\
&-(x_i\xi_j-x_j\xi_i)\xi^p\xi^q\frac{\partial^2}{\partial x^p\partial x^q}.
\end{aligned}
$$
On using the operators
$J_{ij},\ \L\xi,\partial_x\R=\xi^p\frac{\partial}{\partial x^p}$
and
$\L\xi,\partial_\xi\R=\xi^p\frac{\partial}{\partial \xi^p}$,
we write this in the form
\begin{equation}
\begin{aligned}
{\tilde X}_i{\tilde\Xi}_j-{\tilde X}_j{\tilde\Xi}_i&=J_{ij}
+\Big(x_i\frac{\partial}{\partial x^j}-x_j\frac{\partial}{\partial x^i}\Big)\L\xi,\partial_x\R
+\Big(\xi_i\frac{\partial}{\partial x^j}-\xi_j\frac{\partial}{\partial x^i}\Big)\L\xi,\partial_\xi\R\\
&-\Big(\xi_i\frac{\partial}{\partial\xi^j}-\xi_j\frac{\partial}{\partial\xi^i}\Big)\L\xi,\partial_x\R
+\Big(\xi_i\frac{\partial}{\partial x^j}-\xi_j\frac{\partial}{\partial x^i}\Big)
-(x_i\xi_j-x_j\xi_i)\L\xi,\partial_x\R^2.
\end{aligned}
                                         \label{3.5}
\end{equation}

Since $\psi$ is positively homogeneous of degree $\lambda$ in $\xi$ and satisfies \eqref{3.1}, the equations
$$
\L\xi,\partial_x\R\psi=0,\quad \L\xi,\partial_\xi\R\psi=\lambda\psi
$$
hold. Therefore the equality \eqref{3.5} for the function $\psi$ looks as follows:
\begin{equation}
({\tilde X}_i{\tilde\Xi}_j-{\tilde X}_j{\tilde\Xi}_i)\psi=J_{ij}\psi
+(\lambda+1)\Big(\xi_i\frac{\partial}{\partial x^j}-\xi_j\frac{\partial}{\partial x^i}\Big)\psi.
                                         \label{3.6}
\end{equation}

The relation
$$
\xi_i\frac{\partial}{\partial x^j}-\xi_j\frac{\partial}{\partial x^i}=\xi_i X_j-\xi_j X_i
$$
holds on $T{\Sb}^{n-1}$ as is seen from the definition of $X_i$. Therefore \eqref{3.6} implies that
$$
({\tilde X}_i{\tilde\Xi}_j-{\tilde X}_j{\tilde\Xi}_i)\psi=J_{ij}\psi
+(\lambda+1)(\xi_i X_j-\xi_j X_i)\varphi
$$
on $T{\Sb}^{n-1}$. Together with \eqref{3.3}, this gives \eqref{3.2}.
\end{proof}

\begin{lemma} \label{L3.2}
Assume a function $\psi\in C^\infty\big({\Rb}^n\times({\Rb}^n\setminus\{0\})\big)$ to be positively homogeneous of degree $\lambda$ in $\xi$ and to satisfy $\psi(x+t\xi,\xi)=\psi(x,\xi)$ for $t\in{\Rb}$.
Let $\varphi\in C^\infty(T{\Sb}^{n-1})$ be the restriction of $\psi$ to $T{\Sb}^{n-1}$. Then, for every $k=1,2,\dots$, the equality
\begin{equation}
\begin{aligned}
&(J_{i_1j_1}\dots J_{i_kj_k}\psi)|_{T{\Sb}^{n-1}}\\
&=\Big({\mathcal J}_{i_1j_1}-c_1(\lambda,k)\big(\xi_{i_1}X_{j_1}-\xi_{j_1}X_{i_1}\big)\Big)
\Big({\mathcal J}_{i_2j_2}-c_2(\lambda,k)\big(\xi_{i_2}X_{j_2}-\xi_{j_2}X_{i_2}\big)\Big)\\
&\dots\Big({\mathcal J}_{i_kj_k}-c_k(\lambda,k)\big(\xi_{i_k}X_{j_k}-\xi_{j_k}X_{i_k}\big)\Big)\varphi
\end{aligned}
                          \label{3.7}
\end{equation}
holds for any indices $1\le i_1,j_1,\dots,i_k,j_k\le n$, where
\begin{equation}
c_l(\lambda,k)=\lambda-k+l+1.
                                   \label{3.8}
\end{equation}
\end{lemma}

\begin{proof}
The proof is going by induction in $k$.

By Lemma \ref{L3.1},
\begin{equation}
(J_{i_kj_k}\psi)|_{T{\Sb}^{n-1}}=\Big({\mathcal J}_{i_kj_k}-(\lambda+1)\big(\xi_{i_k}X_{j_k}-\xi_{j_k}X_{i_k}\big)\Big)\varphi.
                                   \label{3.9}
\end{equation}
This coincides with \eqref{3.7} for $k=1$ with
\begin{equation}
c_1(\lambda,1)=\lambda+1.
                                   \label{3.10}
\end{equation}

Fix indices $i_k$ and $j_k$ and set $\psi'=J_{i_kj_k}\psi$. Formula \eqref{3.9} can be rewritten as
\begin{equation}
\psi'|_{T{\Sb}^{n-1}}=\Big({\mathcal J}_{i_kj_k}-c_1(\lambda,1)\big(\xi_{i_k}X_{j_k}-\xi_{j_k}X_{i_k}\big)\Big)\varphi.
                                         \label{3.11}
\end{equation}
The function $\psi'$ is homogeneous of degree $\lambda-1$ in $\xi$. Besides this, it satisfies $\L\xi,\partial_x\R\psi'=0$ since the operators $\L\xi,\partial_x\R$ and $J_{kl}$ commute as easily follows from definitions of these operators. By the induction hypothesis,
$$
\begin{aligned}
(J_{i_1j_1}\dots J_{i_{k-1}j_{k-1}}\psi')&|_{T{\Sb}^{n-1}}=
\Big({\mathcal J}_{i_1j_1}-c_1(\lambda-1,k-1)\big(\xi_{i_1}X_{j_1}-\xi_{j_1}X_{i_1}\big)\Big)\dots\\
&\dots\Big({\mathcal J}_{i_{k-1}j_{k-1}}-c_k(\lambda-1,k-1)\big(\xi_{i_{k-1}}X_{j_{k-1}}-\xi_{j_{k-1}}X_{i_{k-1}}\big)\Big)
\big(\psi'|_{T{\Sb}^{n-1}}\big).
\end{aligned}
$$
Substituting the expression \eqref{3.11} into the right-hand side of the last formula and the value $\psi'=J_{i_kj_k}\psi$ into the left-hand side, we obtain \eqref{3.7} with the coefficients
$$
c_1(\lambda,k)=c_1(\lambda\!-\!1,k\!-\!1),\dots,
c_{k-1}(\lambda,k)=c_{k-1}(\lambda\!-\!1,k\!-\!1);\ c_k(\lambda,k)=c_1(\lambda,1).
$$
Together with \eqref{3.10}, last equalities imply \eqref{3.8}
\end{proof}

Let us distinguish the most important case when $k=m+1$ and $\lambda=m-1$.

\begin{corollary} \label{C3.1}
For an integer $m\ge0$, assume a function $\psi\in C^\infty\big({\Rb}^n\times({\Rb}^n\setminus\{0\})\big)$ to be positively homogeneous of degree $m-1$ in $\xi$ and to satisfy $\psi(x+t\xi,\xi)=\psi(x,\xi)$ for $t\in{\Rb}$.
Let $\varphi\in C^\infty(T{\Sb}^{n-1})$ be the restriction of $\psi$ to $T{\Sb}^{n-1}$. Then the equality
\begin{equation}
\begin{aligned}
(J_{i_1j_1}\dots J_{i_{m+1}j_{m+1}}\psi)|_{T{\Sb}^{n-1}}&=
{\mathcal J}_{i_1j_1}
\Big({\mathcal J}_{i_2j_2}-\big(\xi_{i_2}X_{j_2}-\xi_{j_2}X_{i_2}\big)\Big)
\Big({\mathcal J}_{i_3j_3}-2\big(\xi_{i_3}X_{j_3}-\xi_{j_3}X_{i_3}\big)\Big)\\
&\dots\Big({\mathcal J}_{i_{m+1}j_{m+1}}-m\big(\xi_{i_{m+1}}X_{j_{m+1}}-\xi_{j_{m+1}}X_{i_{m+1}}\big)\Big)\varphi
\end{aligned}
                          \label{3.12}
\end{equation}
holds for any indices $1\le i_1,j_1,\dots,i_{m+1},j_{m+1}\le n$.
\end{corollary}

\begin{theorem}[Intrinsic form of Theorem \ref{Th2.1}]\label{Th3.1}
Let
$n\ge3$.
A function
$\varphi\in{\mathcal S}_{\pi(m)}(T{\Sb}^{n-1})$
belongs to the range of the operator \eqref{1.1}
if and only if it satisfies
\begin{equation}
\begin{aligned}
&{\mathcal J}_{i_1j_1}
\Big({\mathcal J}_{i_2j_2}-\big(\xi_{i_2}X_{j_2}-\xi_{j_2}X_{i_2}\big)\Big)
\Big({\mathcal J}_{i_3j_3}-2\big(\xi_{i_3}X_{j_3}-\xi_{j_3}X_{i_3}\big)\Big)\\
&\dots\Big({\mathcal J}_{i_{m+1}j_{m+1}}-m\big(\xi_{i_{m+1}}X_{j_{m+1}}-\xi_{j_{m+1}}X_{i_{m+1}}\big)\Big)\varphi=0
\end{aligned}
                                            \label{3.13}
\end{equation}
for all indices $1\le i_1,j_1,\dots,i_{m+1},j_{m+1}\le n$.
\end{theorem}

\begin{proof}
Necessity. Let a function
$\varphi\in{\mathcal S}_{\pi(m)}(T{\Sb}^{n-1})$
belong to the range of the operator \eqref{1.1}. Define the function
$\psi\in C^\infty\big({\Rb}^n\times({\Rb}^n\setminus\{0\})\big)$
by \eqref{2.5}.
Then
$\psi|_{T{\Sb}^{n-1}}=\varphi$.
By Theorem \ref{Th2.1}, the function
$\psi$
satisfies the John equations \eqref{2.6}
for all indices $1\le i_1,j_1,\dots,i_{m+1},j_{m+1}\le n$. The left-hand side of equation \eqref{3.12} is equal to zero.  Equating the right-hand side of \eqref{3.12} to zero, we arrive to \eqref{3.13}.

Sufficiency. Let a function
$\varphi\in{\mathcal S}_{\pi(m)}(T{\Sb}^{n-1})$
satisfy \eqref{3.13}. We again define
$\psi\in C^\infty\big({\Rb}^n\times({\Rb}^n\setminus\{0\})\big)$
by \eqref{2.4}. Then
$\psi|_{T{\Sb}^{n-1}}=\varphi$ and Corollary \ref{C3.1} applies to the functions $\psi$ and $\varphi$.
By \eqref{3.13}, the right-hand side of \eqref{3.12} is equal to zero. Equating the left-hand side of \eqref{3.12} to zero, we obtain
$(J_{i_1j_1}\dots J_{i_{k-1}j_{k-1}}\psi)|_{T{\Sb}^{n-1}}=0$. Since $\psi$ satisfies
\[
\psi(x+t\xi,\xi)=\psi(x,\xi)\ (t\in\Rb),\quad \psi(x,t\xi)=\mbox{sgn}(t)\psi(x,\xi)\ (0\neq t\in\Rb),
\]
the latter equation implies $J_{i_1j_1}\dots J_{i_{k-1}j_{k-1}}\psi=0$.
Applying Theorem \ref{Th2.1}, we infer that
$\varphi$
belongs to the range of the operator \eqref{1.1}.
\end{proof}

Recall that ${\mathcal S}_{\mathrm sol}({\Rb}^n;S^m{\Rb}^n)$ is the space of smooth solenoidal fast decaying tensor fields defined by \eqref{2.10}.

\begin{lemma}[{\rm Main lemma}] \label{L3.3}
Let $n\ge3$ and $t\in\big(-(n-1)/2,(n-2)/2\big)$. For every integer $m\ge0$ and every real $s$, the following statement is valid.

Let a ``function'' $\varphi\in H^{s+1/2}_{t+1/2,\pi(m)}(T{\Sb}^{n-1})$ satisfy equations \eqref{2.14} in the distribution sense for all $1\le i_1,j_1,\dots,i_{m+1},j_{m+1}\le n$. Assume also that
\begin{equation}
(\varphi,If)_{H^{s+1/2}_{t+1/2}(T{\Sb}^{n-1})}=0\quad\mbox{\rm for every tensor field}\quad f\in {\mathcal S}_{sol}({\Rb}^n;S^m{\Rb}^n).
                                        \label{3.14}
\end{equation}
 Then $\varphi=0$.
\end{lemma}

We emphasize that the integer $r$ does not participate in Lemma \ref{L3.3} (more precisely, $r=0$ in the lemma). It is the main advantage of  Lemma \ref{L3.3} as compared with Theorem \ref{Th2.2}. Even in the case of $r=0$, Lemma \ref{L3.3} has one more  advantage: we do not need to recall the definition of the norm $\|\cdot\|_{H^{(0,s)}_{t,\mathrm{sol}}({\Rb}^n;S^m{\Rb}^n)}$.

The proof of Lemma \ref{L3.3} will be presented in the next section. Now, we present the proof of Theorem \ref{Th2.2} assuming Lemma \ref{L3.3} to be valid.

\begin{proof}[Proof of Theorem \ref{Th2.2}]
We start with the easy proof of the ``only if'' statement of Theorem \ref{Th2.2}.
Given $f\in H^{(r,s)}_{t,\mathrm{sol}}({\Rb}^n;S^m{\Rb}^n)$, set $\varphi=If\in H^{(r,s+1/2)}_{t+1/2,\pi(m)}(T{\Sb}^{n-1})$. Since
${\mathcal S}_{\mathrm{sol}}({\Rb}^n;S^m{\Rb}^n)$ is a dense subspace of $H^{(r,s)}_{t,\mathrm{sol}}({\Rb}^n;S^m{\Rb}^n)$, we can choose a sequence of tensor fields $f_k\in{\mathcal S}_{sol}({\Rb}^n;S^m{\Rb}^n)\ (k=1,2,\dots)$ such that
\begin{equation}
f_k\to f\quad\mbox{in}\quad H^{(r,s)}_{t,\mathrm{sol}}({\Rb}^n;S^m{\Rb}^n)\quad\mbox{as}\quad k\to\infty.
                          \label{3.15}
\end{equation}
Set $\varphi_k=If_k$.
By Theorem \ref{Th3.1}, every function $\varphi_k\in{\mathcal S}_{\pi(m)}(T{\Sb}^{n-1})$ satisfies the equations
\begin{equation}
\begin{aligned}
&{\mathcal J}_{i_1j_1}\Big({\mathcal J}_{i_2j_2}-(\xi_{i_2}X_{j_2}-\xi_{j_2}X_{i_2})\Big)
\Big({\mathcal J}_{i_3j_3}-2(\xi_{i_3}X_{j_3}-\xi_{j_3}X_{i_3})\Big)\dots\\
&\dots\Big({\mathcal J}_{i_{m+1}j_{m+1}}-m(\xi_{i_{m+1}}X_{j_{m+1}}-\xi_{j_{m+1}}X_{i_{m+1}})\Big)\varphi_k=0
\end{aligned}
                          \label{3.16}
\end{equation}
for all $1\le i_1,j_1,\dots,i_{m+1},j_{m+1}\le n$. Of course \eqref{3.15} implies that
$$
\varphi_k\to \varphi\quad\mbox{in}\quad {\mathcal S}'(T{\Sb}^{n-1})\quad\mbox{as}\quad k\to\infty.
$$
Passing to limit as $k\to\infty$ in \eqref{3.16}, we arrive to \eqref{2.14}.

Now, we prove the ``if'' statement of Theorem \ref{Th2.2}.
The following diagram is helpful for better understanding our arguments
\begin{equation}
\begin{aligned}
&H^{(r,s)}_{t,\mathrm{sol}}({\Rb}^n;S^m{\Rb}^n)\quad{\stackrel I\longrightarrow}\quad H^{(r,s+1/2)}_{t+1/2,\pi(m)}(T{\Sb}^{n-1})\\
&\qquad i\downarrow\qquad\qquad\qquad\qquad\qquad\qquad\downarrow j\\
&H^{s}_{t,\mathrm{sol}}({\Rb}^n;S^m{\Rb}^n)\quad{\stackrel I\longrightarrow}\quad  H^{s+1/2}_{t+1/2,\pi(m)}(T{\Sb}^{n-1})
\end{aligned}
                          \label{3.17}
\end{equation}
Here horizontal arrows (the ray transform) are isometric embeddings of Hilbert spaces while vertical arrows are continuous embeddings.
The diagram is commutative, i.e., $iI=Ij$.

Let $({\rm{Ker}}\,J)^{(r,s+1/2)}_{t+1/2}$ be the closed subspace of $H^{(r,s+1/2)}_{t+1/2,\pi(m)}(T{\Sb}^{n-1})$ consisting of ``functions'' $\varphi$ satisfying equations \eqref{3.16} for all $1\le i_1,j_1,\dots,i_{m+1},j_{m+1}\le n$. We denote the range of the operator \eqref{2.11} by $({\rm{Ran}}\,I)^{(r,s+1/2)}_{t+1/2}$. It is also a closed subspace of $H^{(r,s+1/2)}_{t+1/2,\pi(m)}(T{\Sb}^{n-1})$ by the Reshetnyak formula. By the ``only if'' statement of Theorem \ref{Th2.2},
\begin{equation}
({\rm{Ran}}\,I)^{(r,s+1/2)}_{t+1/2}\subset({\rm{Ker}}\,J)^{(r,s+1/2)}_{t+1/2}.
                          \label{3.18}
\end{equation}
We have to prove that actually there is the equality in \eqref{3.18}. To this end we represent $({\rm{Ker}}\,J)^{(r,s+1/2)}_{t+1/2}$ as
\begin{equation}
({\rm{Ker}}\,J)^{(r,s+1/2)}_{t+1/2}=({\rm{Ran}}\,I)^{(r,s+1/2)}_{t+1/2}\oplus\Big(({\rm{Ran}}\,I)^{(r,s+1/2)}_{t+1/2}\big)^\bot,
                          \label{3.19}
\end{equation}
where $\Big(({\rm{Ran}}\,I)^{(r,s+1/2)}_{t+1/2}\big)^\bot$ is the orthogonal complement of $({\rm{Ran}}\,I)^{(r,s+1/2)}_{t+1/2}$ in
$({\rm{Ker}}\,J)^{(r,s+1/2)}_{t+1/2}$ with respect to the scalar product $(\cdot,\cdot)_{H^{s+1/2}_{t+1/2}(T{\Sb}^{n-1})}$. The embedding \eqref{3.18} is the equality if and only if the second summand on the right-hand side of \eqref{3.19} is equal to zero.

Let $\varphi\in\Big(({\rm{Ran}}\,I)^{(r,s+1/2)}_{t+1/2}\big)^\bot$. Then the ``function'' $\varphi$ satisfies equations \eqref{2.14} for all $1\le i_1,j_1,\dots,i_{m+1},j_{m+1}\le n$ and the statement \eqref{3.14} holds for $\varphi$.

Recall that $j$ denotes the right vertical arrow in the diagram \eqref{3.17}. The embedding $j$ obviously preserves the scalar product $(\cdot,\cdot)_{H^{s+1/2}_{t+1/2}(T{\Sb}^{n-1})}$. Therefore \eqref{3.14} implies
$$
(j\varphi,If)_{H^{s+1/2}_{t+1/2}(T{\Sb}^{n-1})}=0\quad\mbox{\rm for every tensor field}\quad f\in {\mathcal S}_{sol}({\Rb}^n;S^m{\Rb}^n).
$$
Besides this, the ``function'' $j\varphi$ satisfies equations \eqref{2.14} for all $1\le i_1,j_1,\dots,i_{m+1},j_{m+1}\le n$ since these equations are understood in the distribution sense. Thus,  the ``function'' $j\varphi\in  H^{s+1/2}_{t+1/2,\pi(m)}(T{\Sb}^{n-1})$ satisfies all hypotheses of Lemma \ref{L3.3}. By the lemma, $j\varphi=0$. Since $j$ is an injective operator, this implies $\varphi=0$.
\end{proof}

\section{Proof of Lemma \ref{L3.3}}

Recall that the weighted $L^2$-space
$L^{2,s}_t({\Rb}^n)$
was defined in \cite[Section 3]{Sh4} for
$s\in\Rb$
and
$t>-n/2$
as the completion of
${\mathcal S}({\Rb}^n)$
with respect to  the norm
$$
\|f\|^2_{L^{2,s}_t({\Rb}^n)}=\int\limits_{{\Rb}^n}|y|^{2t}(1+|y|^2)^{s-t}|f(y)|^2\, d y.
$$

The Hilbert space $L^{2,s}_{t,e}(T{\Sb}^{n-1})$ of even ``functions'' was defined in \cite[Section 3]{Sh4} for $s\in\Rb$
and
$t>-(n-1)/2$
as the completion of
${\mathcal S}_e(T{\Sb}^{n-1})$ (the index $e$ stands for ``even'')
with respect to  the norm
\begin{equation}
\|\varphi\|^2_{L^{2,s}_t(T{\Sb}^{n-1})}=\frac{\Gamma\big(\frac{n-1}{2}\big)}{4\pi^{(n+1)/2}}
\int\limits_{{\Sb}^{n-1}}\int\limits_{\xi^\bot}|y|^{2t}(1+|y|^2)^{s-t}|\varphi(y,\xi)|^2\, d y d \xi.
                                              \label{4.1}
\end{equation}
By \cite[Lemma 3.1]{Sh4}, the Fourier transform
$F:{\mathcal S}_e(T{\Sb}^{n-1})\rightarrow{\mathcal S}_e(T{\Sb}^{n-1})$
extends to the bijective isometry of Hilbert spaces
$$
F:H^s_{t,e}(T{\Sb}^{n-1})\rightarrow L^{2,s}_{t,e}(T{\Sb}^{n-1}).
$$

Quite similarly, the weighted $L^2$-space $L^{2,s}_{t,o}(T{\Sb}^{n-1})$ of odd ``functions'' is defined  for $s\in\Rb$
and
$t>-(n-1)/2$
as the completion of
${\mathcal S}_o(T{\Sb}^{n-1})$ (the index $o$ stands for ``odd'')
with respect to  the same norm \eqref{4.1}.
The Fourier transform
$F:{\mathcal S}_o(T{\Sb}^{n-1})\rightarrow{\mathcal S}_o(T{\Sb}^{n-1})$
extends to the bijective isometry of Hilbert spaces
$$
F:H^s_{t,o}(T{\Sb}^{n-1})\rightarrow L^{2,s}_{t,o}(T{\Sb}^{n-1}).
$$

We will first prove the following statement.

\begin{lemma} \label{L4.1}
Let $m\ge0$ and $n\ge3$ be integers, $s\in\Rb$ and $t\in\big(-(n-1)/2,(n-1)/2\big)$. Given
a ``function'' $\varphi\in H^s_{t,\pi(m)}(T{\Sb}^{n-1})$ satisfying equations \eqref{2.14} in the distribution sense for all $1\le i_1,j_1,\dots,i_{m+1},j_{m+1}\le n$, let $\widehat\varphi\in L^{2,s}_{t,\pi(m)}(T{\Sb}^{n-1})$ be the Fourier transform of $\varphi$. There exists a unique symmetric tensor field $\widehat\Phi=(\widehat\Phi_{i_1\dots i_m})\in L^{2,s}_t({\Rb}^n;S^m{\Rb}^n)$ satisfying
\begin{equation}
y^p\,\widehat\Phi_{pi_2\dots i_m}(y)=0
                                        \label{4.2}
\end{equation}
and such that
\begin{equation}
\widehat\varphi(y,\xi)=\widehat\Phi_{i_1\dots i_m}(y)\xi^{i_1}\dots\xi^{i_m}\quad \big((y,\xi)\in T{\Sb}^{n-1}\big).
                                        \label{4.3}
\end{equation}
\end{lemma}

Before we begin the proof of this lemma, we define a suitable smooth hypersurface $V_n$ in $({\Rb}^n\setminus\{0\})\times({\Rb}^n\setminus\{0\})$ in which the operators appearing in \eqref{2.14} (in the Fourier variables) simplify.

Let
$$
{\Sb}^{n-1}_0=\{(0,\xi)\mid\xi\in{\Sb}^{n-1}\}\subset T{\Sb}^{n-1}
$$
and
$$
V_n=\{(x,\xi)\in{\Rb}^n\times{\Rb}^n\mid x\neq0,\xi\neq0,\L x,\xi\R=0\}\subset({\Rb}^n\setminus\{0\})\times({\Rb}^n\setminus\{0\}).
$$
$V_n$ is a hypersurface in $({\Rb}^n\setminus\{0\})\times({\Rb}^n\setminus\{0\})$. The manifold $V_n$ is diffeomorphic to $(T{\Sb}^{n-1}\setminus{\Sb}^{n-1}_0)\times(0,\infty)$, the diffeomorphism is defined by
$$
{(T{\Sb}^{n-1}\setminus{\Sb}^{n-1}_0)\times(0,\infty)}\longrightarrow V_n,\quad
(x,\xi;\rho)\mapsto(x,\rho\xi).
$$
On the other hand, $T{\Sb}^{n-1}\setminus{\Sb}^{n-1}_0$ is a hypersurface in $V_n$ and there is the {\it projection}
$$
p:V_n\rightarrow T{\Sb}^{n-1}\setminus{\Sb}^{n-1}_0,\quad p(x,\xi)=(x,\xi/|\xi|).
$$

For the manifold $V_n$, we repeat arguments of \cite[Section 4]{Krishnan:Manna:Sahoo:Sharafutdinov:2019} with small modifications. Namely, we define vector fields $\tilde{X'_i},\tilde{\Xi'_i}\ (1\le i\le n)$ on ${\Rb}^n\times({\Rb}^n\setminus\{0\})=\{(x,\xi)\mid\xi\neq0\}$ by
\begin{equation}
\begin{aligned}
\tilde{X'_i}&=\frac{\partial}{\partial x^i}-\frac{1}{|\xi|^2}\,\xi_i\xi^p\frac{\partial}{\partial x^p}\quad(1\le i\le n),\\
\tilde{\Xi'_i}&=\frac{\partial}{\partial \xi^i}-\frac{1}{|\xi|^2}\,x_i\xi^p\frac{\partial}{\partial x^p}
-\frac{1}{|\xi|^2}\,\xi_i\xi^p\frac{\partial}{\partial \xi^p}\quad(1\le i\le n).
\end{aligned}
                          \label{4.4}
\end{equation}
As compared with formulas (4.3) of \cite{Krishnan:Manna:Sahoo:Sharafutdinov:2019}, there is the additional factor $1/|\xi|^2$ at some terms on right-hand sides of \eqref{4.4}. The factor is added to make $\tilde{X'_i}$ and $\tilde{\Xi'_i}$ homogeneous in $\xi$. At points of $T{\Sb}^{n-1}\setminus{\Sb}^{n-1}_0$, the vector fields $\tilde{X'_i}$ and $\tilde{\Xi'_i}$ coincide with $\tilde X_i$ and $\tilde\Xi_i$ respectively.

At every point $(x,\xi)\in V_n$, vectors ${\tilde X'}_i(x,\xi)$ and ${\tilde \Xi'}_i(x,\xi)\ (1\le i\le n)$ are tangent to $V_n$. This follows from the obvious equalities
$$
\tilde{X'_i}\L x,\xi\R=0,\quad \tilde{\Xi'_i}\L x,\xi\R=0\quad (1\le i\le n)
$$
which hold on $V_n$.
Let $X_i$ and $\Xi_i$ be the restrictions of vector fields ${\tilde X'}_i$ and ${\tilde \Xi'}_i$ to the manifold $V_n$ respectively. Thus, $X_i$ and $\Xi_i$ are smooth vector fields on $V_n$ and can be considered as first order differential operators
$$
X_i,\Xi_i:C^\infty(V_n)\rightarrow C^\infty(V_n).
$$
The restrictions of the vector fields $X_i$ and $\Xi_i$ to the submanifold $T{\Sb}^{n-1}\setminus{\Sb}^{n-1}_0\subset V_n$ coincide with the vector fields on $T{\Sb}^{n-1}\setminus{\Sb}^{n-1}_0$ which were denoted by the same notations $X_i$ and $\Xi_i$ before. The coincidence of notations should not imply any misunderstanding since the new vector fields $X_i$ and $\Xi_i$ are natural extensions of old $X_i$ and $\Xi_i$ from the manifold $T{\Sb}^{n-1}\setminus{\Sb}^{n-1}_0$ to $V_n$.

The main advantage of the manifold $V_n$, as compared with $T{\Sb}^{n-1}\setminus{\Sb}^{n-1}_0$, is the following fact. The vector fields
$\L \xi,\partial_\xi\R=\xi^p\frac{\partial}{\partial\xi^p}$ and
$x_i\frac{\partial}{\partial\xi^j}-x_j\frac{\partial}{\partial\xi^i}\ (1\le i,j\le n)$
are tangent to $V_n$ at every point $(x,\xi)\in V_n$. This follows from the equalities
$$
\L \xi,\partial_\xi\R(\L y,\xi\R)=0,\quad \Big(x_i\frac{\partial}{\partial\xi^j}-x_j\frac{\partial}{\partial\xi^i}\Big)\L y,\xi\R=0
$$
which hold at points $(x,\xi)\in V_n$. Thus, we have the well defined first order differential operators
\begin{equation}
\L \xi,\partial_\xi\R:C^\infty(V_n)\rightarrow C^\infty(V_n),\quad
x_i\frac{\partial}{\partial\xi^j}-x_j\frac{\partial}{\partial\xi^i}:C^\infty(V_n)\rightarrow C^\infty(V_n)\quad(1\le i,j\le n).
                          \label{4.5}
\end{equation}

Let ${\mathcal D}'(V_n)$ be the topological vector space of distributions on the manifold $V_n$.
Since differentiation and multiplication of a distribution by a smooth function are well defined,
$$
\L \xi,\partial_\xi\R:{\mathcal D}'(V_n)\rightarrow {\mathcal D}'(V_n),\quad
y_i\frac{\partial}{\partial\xi^j}-y_j\frac{\partial}{\partial\xi^i}:{\mathcal D}'(V_n)\rightarrow {\mathcal D}'(V_n)\quad(1\le i,j\le n)
$$
are well defined differential operators. We note that the latter formula is written in the variables $(y,\xi)$ that are Fourier dual of $(x,\xi)$.

The operators $x_i\Xi_j-x_j\Xi_i:C^\infty(V_n)\rightarrow C^\infty(V_n)$ can be expressed through operators \eqref{4.5}:
\begin{equation}
x_i\Xi_j-x_j\Xi_i=
x_i\frac{\partial}{\partial\xi^j}-x_j\frac{\partial}{\partial\xi^i}-\frac{1}{|\xi|^2}\,(x_i\xi_j-x_j\xi_i)\L \xi,\partial_\xi\R
\quad(1\le i,j\le n).
                          \label{4.6}
\end{equation}
This easily follows from the definition \eqref{4.4}.

\begin{lemma} \label{L4.2}
Let a function $\varphi\in C^\infty(T{\Sb}^{n-1}\setminus{\Sb}^{n-1}_0)$ satisfy the equations
\begin{equation}
\begin{aligned}
&\Big(x_{i_{m+1}}\Xi_{j_{m+1}}\!-\!x_{j_{m+1}}\Xi_{i_{m+1}}\Big)
\Big(x_{i_m}\Xi_{j_m}\!-\!x_{j_m}\Xi_{i_m}+(x_{i_m}\xi_{j_m}\!-\!x_{j_m}\xi_{i_m})\Big)\\
&\Big(x_{i_{m\!-\!1}}\Xi_{j_{m\!-\!1}}\!-\!x_{j_{m\!-\!1}}\Xi_{i_{m\!-\!1}}\!+\!2(x_{i_{m\!-\!1}}\xi_{j_{m\!-\!1}}
\!-\!x_{j_{m\!-\!1}}\xi_{i_{m\!-\!1}})\Big)\dots\\
&\dots\Big(x_{i_1}\Xi_{j_1}\!-\!x_{j_1}\Xi_{i_1}\!+\!m(x_{i_1}\xi_{j_1}\!-\!x_{j_1}\xi_{i_1})\Big)\varphi=0
\end{aligned}
                          \label{4.7}
\end{equation}
for all indices $1\le i_1,j_1,\dots i_{m+1},j_{m+1}\le n$.
Define the function $\psi\in C^\infty(V_n)$ by
\begin{equation}
\psi(x,\xi)=|\xi|^m\varphi\Big(x,\frac{\xi}{|\xi|}\Big).
                          \label{4.8}
\end{equation}
This function satisfies the equations
\begin{equation}
\Big(x_{i_1}\frac{\partial}{\partial\xi^{j_1}}-x_{j_1}\frac{\partial}{\partial\xi^{i_1}}\Big)\dots
\Big(x_{i_{m+1}}\frac{\partial}{\partial\xi^{j_{m+1}}}-x_{j_{m+1}}\frac{\partial}{\partial\xi^{i_{m+1}}}\Big)\psi=0
                          \label{4.9}
\end{equation}
also for all indices.
\end{lemma}

\begin{proof}
Being defined by \eqref{4.8}, the function $\psi\in C^\infty(V_n)$ is positively homogeneous of degree $m$ in $\xi$ and satisfies
$\psi|_{T{\Sb}^{n-1}\setminus{\Sb}^{n-1}_0}=\varphi$.

Let us fix indices $i_1,j_1,\dots,i_{m+1},j_{m+1}$ and define the functions $\psi_1,\psi_2,\dots,\psi_{m+1}\in C^\infty(V_n)$ by the recurrent relations
\begin{equation}
\psi_1=\Big(x_{i_1}\Xi_{j_1}-x_{j_1}\Xi_{i_1}+\frac{m}{|\xi|^2}(x_{i_1}\xi_{j_1}-x_{j_1}\xi_{i_1})\Big)\psi,
                          \label{4.10}
\end{equation}
\begin{equation}
\psi_k=\Big(x_{i_k}\Xi_{j_k}-x_{j_k}\Xi_{i_k}+\frac{m-k+1}{|\xi|^2}(x_{i_k}\xi_{j_k}-x_{j_k}\xi_{i_k})\Big)\psi_{k-1}\quad(k=2,\dots,m+1),
                          \label{4.11}
\end{equation}
where $\Xi_i$ are now understood as differential operators on the manifold $V_n$. Observe that every $\psi_k(x,\xi)$ is positively homogeneous of degree
$m-k$ in $\xi$. Combining \eqref{4.10}--\eqref{4.11}, we see that
$$
\begin{aligned}
\psi_{m+1}=&\Big(x_{i_{m+1}}\Xi_{j_{m+1}}\!-\!x_{j_{m+1}}\Xi_{i_{m+1}}\Big)
\Big(x_{i_m}\Xi_{j_m}\!-\!x_{j_m}\Xi_{i_m}+\frac{1}{|\xi|^2}(x_{i_m}\xi_{j_m}\!-\!x_{j_m}\xi_{i_m})\Big)\\
&\Big(x_{i_{m\!-\!1}}\Xi_{j_{m\!-\!1}}\!-\!x_{j_{m\!-\!1}}\Xi_{i_{m\!-\!1}}\!+\!\frac{2}{|\xi|^2}(x_{i_{m\!-\!1}}\xi_{j_{m\!-\!1}}
\!-\!x_{j_{m\!-\!1}}\xi_{i_{m\!-\!1}})\Big)\!\dots\\
&\dots\Big(x_{i_1}\Xi_{j_1}\!-\!x_{j_1}\Xi_{i_1}\!+\!\frac{m}{|\xi|^2}(x_{i_1}\xi_{j_1}\!-\!x_{j_1}\xi_{i_1})\Big)\psi.
\end{aligned}
$$
Equation \eqref{4.7} is now written as
\begin{equation}
\psi_{m+1}|_{T{\Sb}^{n-1}\setminus{\Sb}^{n-1}_0}=0.
                          \label{4.12}
\end{equation}
Since the function $\psi_{m+1}\in C^\infty(V_n)$ is positively homogeneous of degree $-1$ in $\xi$ and satisfies \eqref{4.12}, it must be identically equal to zero, i.e.,
\begin{equation}
\psi_{m+1}=0.
                          \label{4.13}
\end{equation}

Let us simplify the formula \eqref{4.10}. By \eqref{4.6},
$$
x_{i_1}\Xi_{j_1}-x_{j_1}\Xi_{i_1}=
x_{i_1}\frac{\partial}{\partial\xi^{j_1}}-x_{j_1}\frac{\partial}{\partial\xi^{i_1}}-
\frac{1}{|\xi|^2}\,(x_{i_1}\xi_{j_1}-x_{j_1}\xi_{i_1})\L \xi,\partial_\xi\R.
$$
Substituting this expression into \eqref{4.10}, we obtain
\begin{equation}
\psi_1=\Big(x_{i_1}\frac{\partial}{\partial\xi^{j_1}}-x_{j_1}\frac{\partial}{\partial\xi^{i_1}}
+\frac{1}{|\xi|^2}(x_{i_1}\xi_{j_1}-x_{j_1}\xi_{i_1})(m-\L \xi,\partial_\xi\R)\Big)\psi.
                          \label{4.14}
\end{equation}
Since the function $\psi(x,\xi)$ is a positively homogeneous of degree $m$ in $\xi$, it satisfies
$$
(m-\L \xi,\partial_\xi\R)\psi=0.
$$
Formula \eqref{4.14} is simplified to the following one:
\begin{equation}
\psi_1=\Big(x_{i_1}\frac{\partial}{\partial\xi^{j_1}}-x_{j_1}\frac{\partial}{\partial\xi^{i_1}}\Big)\psi.
                          \label{4.15}
\end{equation}

In the same way we simplify the formula \eqref{4.11} to the following one:
\begin{equation}
\psi_k=\Big(x_{i_k}\frac{\partial}{\partial\xi^{j_k}}-x_{j_k}\frac{\partial}{\partial\xi^{i_k}}\Big)\psi_{k-1}\quad(k=2,\dots,m+1).
                          \label{4.16}
\end{equation}
Combining \eqref{4.15}--\eqref{4.16}, we see that
$$
\psi_{m+1}=\Big(x_{i_{m+1}}\frac{\partial}{\partial\xi^{j_{m+1}}}-x_{j_{m+1}}\frac{\partial}{\partial\xi^{i_{m+1}}}\Big)\dots
\Big(x_{i_1}\frac{\partial}{\partial\xi^{j_1}}-x_{j_1}\frac{\partial}{\partial\xi^{i_1}}\Big)\psi.
$$
Therefore the equation \eqref{4.13} is now written as
\begin{equation}
\Big(x_{i_{m+1}}\frac{\partial}{\partial\xi^{j_{m+1}}}-x_{j_{m+1}}\frac{\partial}{\partial\xi^{i_{m+1}}}\Big)\dots
\Big(x_{i_1}\frac{\partial}{\partial\xi^{j_1}}-x_{j_1}\frac{\partial}{\partial\xi^{i_1}}\Big)\psi=0.
                          \label{4.17}
\end{equation}
This equation holds for any indices. Observe that the order of factors on the left-hand side of \eqref{4.17} does not matter since the operators
$x_i\frac{\partial}{\partial\xi^j}-x_j\frac{\partial}{\partial\xi^i}$ and $x_k\frac{\partial}{\partial\xi^l}-x_l\frac{\partial}{\partial\xi^k}$ commute.
We have thus proved \eqref{4.9}.
\end{proof}
\begin{proof}[Proof of Lemma \ref{L4.1}]
(Compare with the proof of \cite[Lemma 4.10]{Sh4}.)
We restrict the function $\widehat\varphi\in L^{2,s}_{t,\pi(m)}(T{\Sb}^{n-1})$ to the open set $T{\Sb}^{n-1}\setminus{\Sb}^{n-1}_0$ and consider the restriction as a distribution, i.e., $\widehat\varphi\in{\mathcal D}'(T{\Sb}^{n-1}\setminus{\Sb}^{n-1}_0)$, see Proposition \ref{P2.1}.

Given $\widehat\varphi\in{\mathcal D}'(T{\Sb}^{n-1}\setminus{\Sb}^{n-1}_0)$,  similar to what was done in Lemma \ref{L4.2}, we define $\widehat\psi\in{\mathcal D}'(V_n)$ by
\begin{equation}
\widehat\psi(y,\xi)=|\xi|^{m}\widehat\varphi\Big(y,\frac{\xi}{|\xi|}\Big).
                                        \label{4.18}
\end{equation}
The distribution $\widehat\psi(y,\xi)$ is positively homogeneous of degree $m$ in $\xi$. Hereafter we use some basic facts of theory of homogeneous distributions. Theory of homogeneous distributions on ${\Rb}^n$ is presented in \cite[Section 3.2]{He}. In our case, $\widehat\psi(y,\xi)$ is a homogeneous distribution of $\xi$ depending on the additional argument $y$. We use only facts of the theory which can be generalized to distributions with a parameter.

By Lemma \ref{L4.2}, the equations \eqref{2.14} are written in terms of $\widehat\psi$ as follows:
\begin{equation}
\Big(y_{i_1}\frac{\partial}{\partial \xi^{j_1}}-y_{j_1}\frac{\partial}{\partial \xi^{i_1}}\Big)\dots
\Big(y_{i_{m+1}}\frac{\partial}{\partial \xi^{j_{m+1}}}-y_{j_{m+1}}\frac{\partial}{\partial \xi^{i_{m+1}}}
\Big)\widehat\psi=0\quad(1\le i_1,j_1,\cdots, i_{m+1},j_{m+1}\le n).
                          \label{4.19}
\end{equation}
Lemma \ref{L4.2} is formulated for smooth functions. Nevertheless, it is true for distributions as well. The proof is the same with minor modifications.

The {\it projection}
\begin{equation}
q:V_n\to{\Rb}^n\setminus\{0\},\quad q(y,\xi)=y
                          \label{4.20}
\end{equation}
is a smooth fiber bundle with fibers $q^{-1}(y)=y^\bot\setminus\{0\}$ diffeomorphic to ${\Rb}^{n-1}\setminus\{0\}$. It determines the linear continuous operator
$$
q^*:{\mathcal D}'({\Rb}^n\setminus\{0\})\to{\mathcal D}'(V_n),\quad q^*f=f\circ q.
$$
See \cite[Section 6.1]{He} for the definition of the composition of a distribution with a smooth map.

Let us consider first order differential operators
$$
A_{ij}=y_i\frac{\partial}{\partial \xi^j}-y_j\frac{\partial}{\partial \xi^i}:
C^\infty(V_n)\to C^\infty(V_n)\quad (1\le i<j\le n)
$$
as vector fields on the manifold $V_n$. These vector fields are tangent to fibers of the bundle \eqref{4.20}. At every point
$(y,\xi)\in V_n$,
the linear hull of the vectors $A_{ij}(y,\xi)$ coincided with the tangent space $T_{(y,\xi)}\big(q^{-1}(y)\big)$ of the fiber $q^{-1}(y)$.

For an arbitrary point
$(y_0,\xi_0)\in V_n$,
we can choose a neighborhood
$U\subset V_n$ of $(y_0,\xi_0)$
and local coordinates
$(y^1,\dots,y^{n-1},\xi^1,\dots,\xi^{n-1})$
defined in
$U$
so that the restriction of the projection \eqref{4.20} to
$U$
is defined in these coordinates by
$q(y^1,\dots,y^{n-1},\xi^1,\dots,\xi^{n-1})=(y^1,\dots,y^{n-1})$.
Moreover, we can assume without lost of generality that the coordinate map is a diffeomorphism of
$U$
onto
${\Rb}^{n-1}\times ({\Rb}^{n-1}\setminus\{0\})$.
Equations \eqref{4.19} are written in these coordinates as
\begin{equation}
\frac{\partial^{m+1}\widehat\psi}{\partial \xi^{\alpha_1}\dots\partial \xi^{\alpha_{m+1}}}=0\quad(1\le \alpha_1,\dots,\alpha_{m+1}\le n-1).
                                         \label{4.21}
\end{equation}
The summation over repeated Greek indices is performed from 1 to $n-1$.
Let us use the following easy statement 

{\it If a distribution
$\widehat\psi\in{\mathcal D}'\big({\Rb}^{n-1}\times ({\Rb}^{n-1}\setminus\{0\})\big)\ (n\ge3)$
on
${\Rb}^{n-1}\times ({\Rb}^{n-1}\setminus\{0\})=\{(y^1,\dots,y^{n-1},\xi^1,\dots,\xi^{n-1})\}$
satisfies equations \eqref{4.21}, then it is a homogeneous polynomial of degree $m$ in $\xi$,
i.e., there exist distributions
$\widehat\Phi_{\alpha_1\dots\alpha_m} \in{\mathcal D}'({\Rb}^{n-1})$
such that
\begin{equation}
\widehat\psi=(\tilde q\,{}^*\widehat\Phi_{\alpha_1\dots\alpha_m})\xi^{\alpha_1}\dots\xi^{\alpha_m},
                                         \label{4.22}
\end{equation}
where
$\tilde q:{\Rb}^{n-1}\times ({\Rb}^{n-1}\setminus\{0\})\rightarrow{\Rb}^{n-1}$ is defined by $\tilde q(y,\xi)=y$}.

Applying this statement, we obtain the representation \eqref{4.22} for $\widehat\psi$ in some local coordinates. Returning to Cartesian coordinates in ${\Rb}^n$, we write the representation in the form
\begin{equation}
\widehat\psi= (q^*\widehat\Phi_{i_1\dots i_m})\xi^{i_1}\dots\xi^{i_m}
                                         \label{4.23}
\end{equation}
with $\widehat\Phi_{i_1\dots i_m}\in{\mathcal D}'({\Rb}^n\setminus\{0\})$. The summation over repeated Latin indices in \eqref{4.23} is performed from 1 to $n$. We will also write \eqref{4.23} in the traditional form
\begin{equation}
\widehat\psi(y,\xi)= \widehat\Phi_{i_1\dots i_m}(y)\xi^{i_1}\dots\xi^{i_m}\quad(y\in{\Rb}^n\setminus\{0\},\xi\in y^\bot\setminus\{0\})
                                         \label{4.24}
\end{equation}
although it is understood in the distribution sense. Of course we can assume $\widehat\Phi_{i_1\dots i_m}$ to be symmetric in all indices, i.e., $\widehat\Phi=(\widehat\Phi_{i_1\dots i_m})\in{\mathcal D}'({\Rb}^n\setminus\{0\};S^m{\Rb}^n)$.

The tensor field $\widehat\Phi$ is not uniquely determined by \eqref{4.24} since the equation \eqref{4.24} holds for $\xi\in y^\bot$ only (in the distribution sense: the equation \eqref{4.23} holds on $V_n$). Nevertheless, we can state: there exists a unique tensor field
$\widehat\Phi=(\widehat\Phi_{i_1\dots i_m})\in{\mathcal D}'({\Rb}^n\setminus\{0\};S^m{\Rb}^n)$ satisfying \eqref{4.24} and \eqref{4.2}, where \eqref{4.2} is also understood in the distribution sense. Indeed, let  $\widehat{\Phi'}\in{\mathcal D}'({\Rb}^n\setminus\{0\};S^m{\Rb}^n)$ be an arbitrary tensor field satisfying \eqref{4.24}. Define
$\widehat\Phi=(\widehat\Phi_{i_1\dots i_m})\in{\mathcal D}'({\Rb}^n\setminus\{0\};S^m{\Rb}^n)$ by
$$
\widehat\Phi_{i_1\dots i_m}(y)=\Big(\delta_{i_1}^{j_1}-\frac{y_{i_1}y^{j_1}}{|y|^2}\Big)\dots
\Big(\delta_{i_m}^{j_m}-\frac{y_{i_m}y^{j_m}}{|y|^2}\Big)\widehat{\Phi'}_{j_1\dots j_m}(y)\quad(0\neq y\in{\Rb}^n),
$$
where $\delta_i^j$ is the Kronecker tensor. The latter formula can be also understood in the distribution sense. The tensor field $\widehat\Phi$ satisfies \eqref{4.24} and \eqref{4.2} and is uniquely determined by these relations. In terms of \cite[Lemma 2.6.1]{Sharafutdinov:Book},
$\widehat\Phi$ is the {\it tangential component} of $\widehat{\Phi'}$.

The projection \eqref{4.20} is the composition $q=\pi\circ p$, where the smooth maps
$$
V_n\ {\stackrel p\longrightarrow}\ T{\Sb}^{n-1}\setminus{\Sb}^{n-1}_0\ {\stackrel \pi\longrightarrow}\ {\Rb}^n\setminus\{0\}
$$
are defined by $p(y,\xi)=(y,\xi/|\xi|)$ and $\pi(y,\xi)=y$. Actually $p$ and $\pi$ are also smooth fiber bundles. The formula \eqref{4.23} can be written as
$$
\widehat\psi= |\xi|^m(p^*\pi^*\widehat\Phi_{i_1\dots i_m})\frac{\xi^{i_1}}{|\xi|}\dots\frac{\xi^{i_m}}{|\xi|}.
$$
Since $\frac{\xi^i}{|\xi|}=\xi^i\circ p$, the formula can be also written as
\begin{equation}
\widehat\psi= |\xi|^m p^*\big((\pi^*\widehat\Phi_{i_1\dots i_m})\xi^{i_1}\dots\xi^{i_m}\big).
                                         \label{4.25}
\end{equation}
By the definition \eqref{4.18},
$$
\widehat\psi= |\xi|^m p^*\widehat\varphi.
$$
This gives together with \eqref{4.25}
$$
p^*\widehat\varphi=p^*\big((\pi^*\widehat\Phi_{i_1\dots i_m})\xi^{i_1}\dots\xi^{i_m}\big).
$$
Since $p^*$ is an injective operator, this implies
$$
\widehat\varphi=(\pi^*\widehat\Phi_{i_1\dots i_m})\xi^{i_1}\dots\xi^{i_m}
$$
or in the traditional form
\begin{equation}
\widehat\varphi(y,\xi)=\widehat\Phi_{i_1\dots i_m}(y)\xi^{i_1}\dots\xi^{i_m}\quad\big((y,\xi)\in T{\Sb}^{n-1}\setminus{\Sb}^{n-1}_0\big).
                                         \label{4.26}
\end{equation}

We have thus proved that the statement \eqref{4.3} of Lemma \ref{L4.1} holds on $T{\Sb}^{n-1}\setminus{\Sb}^{n-1}_0$. By the hypothesis of the lemma, $\widehat\varphi\in L^{2,s}_{t,\pi(m)}(T{\Sb}^{n-1})$. Together with \eqref{4.26}, this easily implies that
$\widehat\Phi\in L^{2,s}_t({\Rb}^n\setminus\{0\};S^m{\Rb}^n)$ and the formula \eqref{4.3} holds in the conventional sense at almost all points $(y,\xi)\in T{\Sb}^{n-1}$.
\end{proof}

\begin{proof}[Proof of Lemma \ref{L3.3}]
By the definition of the $H^{s+1/2}_{t+1/2}(T{\Sb}^{n-1})$-norm, the equation \eqref{3.14} is written as follows:
\begin{equation}
\int\limits_{{\Sb}^{n-1}}\int\limits_{\xi^\bot}|y|^{2t+1}(1+|y|^2)^{s-t}\widehat\varphi(y,\xi)\overline{\widehat{If}(y,\xi)}\,dyd\xi=0
\quad\big(f\in{\mathcal S}_{sol}({\Rb}^n;S^m{\Rb}^n)\big).
                                        \label{4.27}
\end{equation}

Applying Lemma \ref{L4.1} to the function $\widehat\varphi\in L^{2,s+1/2}_{t+1/2.\pi(m)}(T{\Sb}^{n-1})$, we can state the existence of a tensor field $\widehat\Phi\in  L^{2,s+1/2}_{t+1/2}({\Rb}^n;S^m{\Rb}^n)$ satisfying \eqref{4.2} and \eqref{4.3}. With the help of \eqref{4.3}, the equation \eqref{4.27} becomes
\begin{equation}
\int\limits_{{\Sb}^{n-1}}\int\limits_{\xi^\bot}|y|^{2t+1}(1+|y|^2)^{s-t}
{\widehat\Phi}{}^{i_1\dots i_m}(y)\overline{\widehat{If}(y,\xi)}\,\xi_{i_1}\dots\xi_{i_m}\,dyd\xi=0.
                                        \label{4.28}
\end{equation}

By \eqref{2.3},
$$
\widehat{If}(y,\xi)=(2\pi)^{1/2}\widehat f_{i_1\dots i_m}(y)\xi^{i_1}\dots\xi^{i_m}\quad\big((y,\xi)\in T{\Sb}^{n-1}\big).
$$
Substituting this expression into \eqref{4.28}, we obtain
$$
\int\limits_{{\Sb}^{n-1}}\int\limits_{\xi^\bot}|y|^{2t+1}(1+|y|^2)^{s-t}
{\widehat\Phi}{}^{i_1\dots i_m}(y)\overline{\widehat f^{j_1\dots j_m}(y)}\,\xi_{i_1}\dots\xi_{i_m}\xi_{j_1}\dots\xi_{j_m}\,dyd\xi=0.
$$
Changing the order of integrations with the help of \cite[Lemma 2.15.3]{Sharafutdinov:Book}, we write this in the form
\begin{equation}
\int\limits_{{\Rb}^n}|y|^{2t}(1+|y|^2)^{s-t}{\widehat\Phi}{}^{i_1\dots i_m}(y)\overline{\widehat f^{j_1\dots j_m}(y)}
\bigg[\int\limits_{{\Sb}^{n-1}\cap\xi^\bot}\xi_{i_1}\dots\xi_{i_m}\xi_{j_1}\dots\xi_{j_m}\,d^{n-2}\xi\bigg]\,dy=0.
                                        \label{4.29}
\end{equation}
By \cite[Lemma 2.15.4]{Sharafutdinov:Book},
$$
\int\limits_{{\Sb}^{n-1}\cap\xi^\bot}\xi_{i_1}\dots\xi_{i_m}\xi_{j_1}\dots\xi_{j_m}\,d^{n-2}\xi
=c_{m,n}\varepsilon^m_{i_1\dots i_mj_1\dots j_m}(y)
$$
with some positive constant $c_{m,n}$, where
\begin{equation}
\varepsilon_{ij}(y)=\delta_{ij}-y_iy_j/|y|^2
                                        \label{4.30}
\end{equation}
and $\varepsilon^m$ is the $m^{\mathrm{th}}$ symmetric power of $\varepsilon$. The formula \eqref{4.29} becomes
\begin{equation}
\int\limits_{{\Rb}^n}|y|^{2t}(1+|y|^2)^{s-t}\varepsilon^m_{i_1\dots i_mj_1\dots j_m}(y)
{\widehat\Phi}{}^{i_1\dots i_m}(y)\overline{\widehat f^{j_1\dots j_m}(y)}\,dy=0.
                                        \label{4.31}
\end{equation}
This equality holds for every $f\in {\mathcal S}_{\mathrm{sol}}({\Rb}^n;S^m{\Rb}^n)$.

We write the equation \eqref{4.31} in the form
\begin{equation}
\int\limits_{{\Rb}^n}\varepsilon^m_{i_1\dots i_mj_1\dots j_m}(y)
\Big(|y|^{t+1/2}(1+|y|^2)^{(s-t)/2}{\widehat\Phi}{}^{i_1\dots i_m}(y)\Big)
\Big(|y|^{t-1/2}(1+|y|^2)^{(s-t)/2}\overline{\widehat f^{j_1\dots j_m}(y)}\Big)dy=0.
                                        \label{4.32}
\end{equation}
Since the tensor field $\widehat\Phi$ belongs to  $L^{2,s+1/2}_{t+1/2}({\Rb}^n;S^m{\Rb}^n)$ and satisfies \eqref{4.2}, the tensor field
$\tilde\Phi=|y|^{t+1/2}(1+|y|^2)^{(s-t)/2}{\widehat\Phi}$ belongs to $L^2({\Rb}^n;S^m{\Rb}^n)$ and satisfies
\begin{equation}
y_p\tilde\Phi{}^{pj_2\dots j_m}(y)=0.
                                        \label{4.33}
\end{equation}
Recall that $2t-1>-n$ by the hypothesis $t\in\big(-(n-1)/2,(n-2)/2\big)$ of Lemma \ref{L3.3}. Since $f\in{\mathcal S}({\Rb}^n;S^m{\Rb}^n)$ is a solenoidal tensor field, the tensor field $g=|y|^{t-1/2}(1+|y|^2)^{(s-t)/2}\widehat f$ belongs to $L^2({\Rb}^n;S^m{\Rb}^n)$ and satisfies
\begin{equation}
y_pg^{pj_2\dots j_m}(y)=0.
                                        \label{4.34}
\end{equation}
The equation \eqref{4.32} is written in terms of $\tilde\Phi$ and $g$ as follows:
\begin{equation}
\int\limits_{{\Rb}^n}\varepsilon^m_{i_1\dots i_mj_1\dots j_m}(y)
{\tilde\Phi}{}^{i_1\dots i_m}(y)\overline{g^{j_1\dots j_m}(y)}\,dy=0.
                                        \label{4.35}
\end{equation}
In terms of \cite[Section 4]{HORF_Work}, \eqref{4.33} and \eqref{4.34} mean that $\tilde\Phi$ and $g$ are {\it tangential} tensor fields.
Let $L^2_\top({\Rb}^n;S^m{\Rb}^n)$ be the space of all such tensor fields.

Since the tensor fields $\tilde\Phi$ and $g$ belong to $L^2_\top({\Rb}^n;S^m{\Rb}^n)$,
the second summand on the right-hand side of \eqref{4.30} gives no contribution to the integral \eqref{4.35}, i.e. $\varepsilon$ can be replaced with the Kronecker tensor $\delta$ in \eqref{4.35}. We thus arrive at the equation
\begin{equation}
\int\limits_{{\Rb}^n}\delta^m_{i_1\dots i_mj_1\dots j_m}(y)
{\tilde\Phi}{}^{i_1\dots i_m}(y)\overline{g^{j_1\dots j_m}(y)}\,dy=0.
                                        \label{4.36}
\end{equation}

Now, we apply some arguments of \cite[Section 2.15]{Sharafutdinov:Book} to equation \eqref{4.36}. Indeed, the integrand in \eqref{4.35} is of the same structure as that in \cite[formula 2.15.11]{Sharafutdinov:Book}. The latter formula was transformed in \cite[Section 2.15]{Sharafutdinov:Book} to a form similar to \eqref{4.36}. We need only the statement following from further arguments presented in \cite[Section 2.15]{Sharafutdinov:Book}: the integral \eqref{4.36} defines a positive scalar product on the space $L^2_\top({\Rb}^n;S^m{\Rb}^n)$.

Being valid for every $g\in L^2_\top({\Rb}^n;S^m{\Rb}^n)$, the equation \eqref{4.36} implies $\tilde\Phi=0$. Together with the equality
$\tilde\Phi=|y|^{t+1/2}(1+|y|^2)^{(s-t)/2}{\widehat\Phi}$, this gives ${\widehat\Phi}=0$. With the help of \eqref{4.26}, the latter equality implies $\widehat\varphi=0$ and $\varphi=0$.
\end{proof}

{\bf Remark.} The hypothesis $n\ge3$ of Lemma \ref{L3.3} was used in our arguments as follows. It is important that ${\Rb}^{n-1}\setminus\{0\}$ is a connected space for the validity of \eqref{4.22}. In the case of $n=2$, the right-hand side of \eqref{4.22} can be represented by two different polynomials on two connected components of ${\Rb}^1\setminus\{0\}$.
\section*{Acknowledgments}
The first author would like to thank the Isaac Newton Institute for Mathematical Sciences, Cambridge, UK, for support and hospitality during the semester programme, \emph{Rich and Nonlinear Tomography - a multidisciplinary approach} in 2023, supported by EPSRC Grant Number EP/R014604/1. The work of the second author was performed according to the Russian Government research assignment for IM~SB~RAS, project FWNF-2022-0006.
\bibliographystyle{alpha}

\end{document}